\documentclass[sn-mathphys]{JORSC-jnl}

\jyear{2023}%

\theoremstyle{thmstyleone}%
\newtheorem{theorem}{Theorem}
\newtheorem{corollary}[theorem]{Corollary}

\theoremstyle{thmstyletwo}%
\newtheorem{example}{Example}%
\newtheorem{remark}{Remark}%

\theoremstyle{thmstylethree}%
\newtheorem{definition}{Definition}%
\newtheorem{OP}{Open Problem}
\newtheorem{assumption}{Assumption}

\usepackage{subfigure}
\usepackage{graphicx}
\usepackage{epstopdf}

\usepackage{ulem}

\usepackage{array,contour}

\newcommand{\g}{\boldsymbol{g}}
\newcommand{\dd}{\boldsymbol{d}}
\newcommand{\rr}{\boldsymbol{r}}
\newcommand{\q}{\boldsymbol{q}}
\newcommand{\PP}{\boldsymbol{P}}

\allowdisplaybreaks[4]

\begin{document}

\title[Derivative-Free Optimization with Transformed Objectives]{Derivative-Free Optimization with Transformed Objective Functions (DFOTO) and the Algorithm Based on the Least Frobenius Norm Updating Quadratic Model}


\author*[1,2]{\fnm{Peng-Cheng} \sur{Xie}}\email{xpc@lsec.cc.ac.cn}

\author[1,2]{\fnm{Ya-Xiang} \sur{Yuan}}\email{yyx@lsec.cc.ac.cn}

\affil*[1]{\orgdiv{State Key Laboratory of Scientific/Engineering Computing, Institute of Computational Mathematics and Scientific/Engineering Computing}, \orgname{Academy of Mathematics and Systems Science, Chinese Academy of Sciences}, \orgaddress{\street{Zhong Guan Cun Donglu 55}, 
\postcode{100190}, \state{Beijing}, \country{China}}}

\affil*[2]{
\orgname{University of Chinese Academy of Sciences}, \orgaddress{\street{Zhong Guan Cun Donglu 55}, 
\postcode{100190}, \state{Beijing}, \country{China}}}


\abstract{Derivative-free optimization (DFO) problems are optimization problems where the derivative information is unavailable. The least Frobenius norm updating quadratic interpolation model function is one of the essential under-determined model functions for model-based derivative-free trust-region methods. This article proposes derivative-free optimization with transformed objective functions (DFOTO) and gives a model-based trust-region method with the least Frobenius norm model. The model updating formula is based on Powell's formula and can be easily implemented. The method shares the same framework with those for problems without transformations, and its query scheme is given. We propose the definitions related to optimality-preserving transformations to understand the interpolation model in our method when minimizing transformed objective functions. We prove the existence of model optimality-preserving transformations beyond translation transformations. The necessary and sufficient condition for such transformations is given. An interesting discovery is that, as a fundamental transformation, the affine transformation with a (non-trivial) positive multiplication coefficient is not model optimality-preserving. We also analyze the corresponding least Frobenius norm updating model and its interpolation error when the objective function is affinely transformed. The convergence property of a provable algorithmic framework containing the least Frobenius norm updating quadratic model for minimizing transformed objective functions is given. Numerical results show that our method can successfully solve most test problems with objective optimality-preserving transformations, even though some of such transformations will change the optimality of the model function. To our best knowledge, this is the first work providing the model-based derivative-free algorithm and analysis for transformed problems with the function evaluation oracle. This article also proposes the ``moving-target'' optimization problem as an open problem.}



\maketitle

\section{Introduction}

Most optimization methods for unconstrained optimization need to use the derivative of the objective function. However, in practice, it is frequent that the objective function is costly to compute, and its derivatives are not available. A typical example is that the objective function is not expressed by an analytic function but obtained through a ``black box'', such as a chemical process or computer simulation. Optimization of this type is called derivative-free optimization. Derivative-free optimization methods are numerical methods in which no first-order or higher-order derivatives are required. For more details on derivative-free optimization, one can see the book of Conn, Scheinberg and Vicente \cite{conn2009introduction}, and Audet and Hare's book \cite{audet2017derivative}. 

\subsection{Derivative-free optimization}

Derivative-free optimization problems appear very often in practice. Applications of derivative-free optimization are widely seen in engineering fields. Some examples are tuning of algorithmic parameters \cite{Audet2006}, optimization of neural networks \cite{2019AlyUEMCON}, automatic error analysis \cite{Higham2002}, dynamic pricing \cite{levina2009dynamic} and optimal design in engineering design \cite{LiandXie}. 

Derivative-free methods can be of different types: for example, direct-search methods, line-search methods, model-based methods, heuristic algorithms, and so on.
\begin{sloppypar}
Direct-search methods include pattern search methods, simplex methods, directional direct-search methods, mesh adaptive direct-search methods, and so on. For example, there are Hooke-Jeeves method \cite{hooke1961direct},  Nelder-Mead method \cite{Nelder1965}, modified simplex method \cite{Tseng1999}, generating set search method \cite{Kolda2003}, and the  nonlinear optimization with the MADS algorithm (NOMAD) \cite{nomad}. Another type of derivative-free methods is the line-search method without derivatives. According to different choices of the searching directions, there are basically three kinds of line-search methods: alternating directions method (e.g., Rosenbrock method \cite{rosenbrock1960automatic}, and the modified ones \cite{swann1972direct}), conjugate directions method \cite{smith1962automatic,powell1964efficient}, and methods based on the approximate gradient (e.g., finite difference quasi-Newton methods 
\cite{stewart1967modification,gill1972quasi,Gill1983}). Methods based on the random gradient approximation have also been discussed \cite{Nesterov2017,duchi2015optimal,Scheinberg2022,zhigljavsky2012theory,berahas2022theoretical}. 
\end{sloppypar}

Model-based methods are important types of derivative-free methods. To be specific, quadratic interpolation model methods, under-determined quadratic interpolation model methods, and regression model methods all use the polynomial models \cite{Winfield1969,Powell2003,Leastf,powell2006newuoa,Conn2008b,Conn2009a,xie2023h2,xie2dmosb,xiesufficient}. Another model-based derivative-free method is the radial basis function interpolation model method \cite{Mattias2000,ORBIT}. {There are also combination methods based on models, such as the method combining the trust-region method and line-search method \cite{xieyuancombination}.} {There are also discussions on the trust-region methods based on probabilistic models \cite{bandeira2014convergence,gratton2018complexity}.}

In addition, most heuristic algorithms do not use derivative information as well, and simulated annealing algorithms \cite{SimulatedAnnealing01} and genetic algorithms \cite{audet2017derivative} are in this type. There are still some other derivative-free methods, such as Bayesian optimization methods \cite{Bayes01}, methods for stochastic optimization \cite{gratton2015direct, bandeira2014convergence,ghadimi2013stochastic} and global optimization \cite{larson2016batch}. Derivative-free methods for special problems also exist, of which an example is the least-squares minimization \cite{zhang2010derivative, cartis2019derivative}, {and another example is the optimization with  special constraints such as the ellipsoid constraint \cite{Xie_2023_ellipsoid}}. Some surveys, such as the paper of  Larson, Menickelly and Wild \cite{larson2019derivative}, Zhang's survey \cite{zhang2021} and Rios and Sahinidis's work \cite{rios2013derivative}, review multiple types of derivative-free optimization methods.

Besides, quite a few excellent software implementations exist for DFO problems. Examples of them are CMA-ES \cite{CMA-ES}, DFO \cite{new-2}, IMFIL \cite{new-4} and so on. The late professor M. J. D. Powell proposed TOLMIN \cite{TOLMIN}, COBYLA \cite{COBYlA}, UOBYQA \cite{UOBYQA}, NEWUOA \cite{powell2006newuoa}, BOBYQA \cite{Powell2009}, and LINCOA \cite{powell2015fast}. There are also DFO-LS \cite{cartis2019improving} and DFBGN \cite{cartis2023scalable}. Recently, Ragonneau and Zhang designed the PDFO, a cross-platform interface for Powell's derivative-free optimization solvers \cite{Ragonneau_Zhang_2021}, and COBYQA \cite{Ragonneau_2022}. We implemented the MATLAB versions of NEWUOA \cite{Xie_2023_NEWUOA_Matlab} and BOBYQA \cite{Xie_2023_BOBYQA_Matlab}, and their Python versions. Zhang's PRIMA \cite{Zhang_2023_PRIMA} provides the reference implementation for Powell's methods with modernization and amelioration. 

\subsection[DFOTO]{Derivative-free optimization with transformed objective functions (DFOTO)}
\label{section1.2} 
This article will focus on solving derivative-free optimization problems with transformed objective functions. The unconstrained derivative-free optimization problems with transformed objective functions\footnote{Constrained problems will be considered in the future. Our results can be extended to some constrained problems in a straightforward way. DFOTO can also be extended to the general transformed optimization using derivative information.} proposed in this article have the general form  
\begin{equation}
\label{transformproblem}
\min_{\boldsymbol{x} \in {\mathbb{R}}^n} \ f(\boldsymbol{x}),
\end{equation}
in which a black box can provide the function values at $m$ points at one iteration/query step, and the query oracle is given in Assumption \ref{Qassumption}. In addition, the $m$ queried points at the same iteration/query step share the same transformation of the objective function when updating. It transforms $f$ to $f_k:=T_k\circ f$, and the transformation $T_k$ only depends on the current ($k$-th) step, with the definition as Definition \ref{def-transformation}. Notice that what we want to minimize finally is still the original objective function $f$. 
\begin{assumption}[Query oracle of DFOTO]
\label{Qassumption}
One query or evaluation obtains the output function values at $m$ points, and this batch of queried points can be chosen by the optimization algorithms. Fig. \ref{fig1_query} illustrates the query oracle.

\begin{figure}[hbtp]
\centering
\includegraphics[width=1\linewidth]{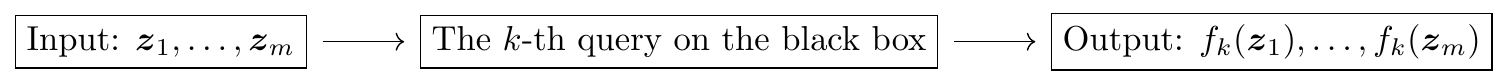}
\caption{Query oracle of DFOTO: the $k$-th query, for the function values of $\boldsymbol{z}_1,\cdots,\boldsymbol{z}_m$\label{fig1_query}}
\end{figure}

\end{assumption}

The query oracle in Assumption \ref{Qassumption} has two essential features. One is that a set of points is queried simultaneously, and the other one is that the simultaneously queried points share the same transformation. Notice that the query in derivative-free optimization usually can not be done in a very short time or low cost, and hence we also call it expensive optimization. Such query oracle usually corresponds to the batch interactive evaluating/simulation mechanism. In addition, we will give more application examples {before the end of Section \ref{section1.2}}. It should be pointed out that, to our best knowledge, although DFOTO has broad applications, the concept and related studies have not been explored deeply or  specifically. Limited algorithms, especially the model-based ones, have been designed or analyzed for such problems. 
This article aims to propose the problem with transformation and to give some preliminary results when using model-based algorithms with a kind of under-determined quadratic model to solve it. We also answer how the least Frobenius norm updating model and the corresponding algorithm are affected by the transformed function value.

\begin{definition}
\label{def-transformation}
Let $T$ be the transformation from ${\mathbb{R}}$ to ${\mathbb{R}}$, and we denote the function $T\circ f$ as the transformed function, which satisfies \((T \circ f)(\boldsymbol{x})=T(f(\boldsymbol{x}))\) for a given function $f$ and arbitrary $\boldsymbol{x}\in {\mathbb{R}}^n$.
\end{definition} 

Every transformation discussed in the following is a transformation from ${\mathbb{R}}$ to ${\mathbb{R}}$. The transformation has no more other requests generally. Transformed objective functions play important roles in stochastic, noisy, or private derivative-free/black-box optimization\footnote{The general form of the unconstrained private black-box optimization problem can be formulated as 
\(
\min_{\boldsymbol{x} \in {\mathbb{R}}^n}\ f(\boldsymbol{x}),
\) 
where $f$ is a private black-box function. The evaluation of $f$ is expensive, and its output value is encrypted to $f_k$ by adding noises, and the true function value of $f$ at each corresponding point is private.}. 
For example, in private black-box optimization, different transformations can be regarded as different encryptions formed by adding different noises according to the theory of differential privacy \cite{privacy13,privacy18,Dwork2006,privacy37,Kasiviswanathan2008,privacy35}. Section \ref{TWT} will present details about solving a special kind of private engineering optimal design problem, which is exactly an example. Another example is the cloud-based distributed optimization problem, which aims to minimize the local and cloud-based composite objective functions while protecting the privacy of the corresponding objective functions \cite{inproceedingsDP}. Besides, there are also private black-box optimization problems in personal health fields \cite{liu2015secure}. Moreover, some work about the differentially private Bayesian optimization has been discussed by Kusner et al. \cite{diffpri}. 

There are multiple applications of derivative-free optimization with transformed objective functions, among which the private black-box optimization serves as a mere instance. For example, the problems with the regularization function, where the coefficients of the regularization term change with the iteration, belong to derivative-free optimization problems with transformed objective functions. A derivative-free trust-region algorithm for composite nonsmooth optimization with a regularization function has been proposed by Grapiglia, Yuan and Yuan \cite{Grapiglia2016}. For the discussion on minimizing transformed objective functions, one may recall the work of Deng and Ferris \cite{Deng06}, which provides a discussion on the algorithm UOBYQA \cite{UOBYQA} for minimizing stochastic objective functions.

In addition to illustrating how the interpolation model change when there are transformations or perturbations in the objective function at each iteration (e.g., the original objective function at a different iteration is transformed by a different affine transformation), we also focus on whether the minimizer of the trust-region subproblem of the quadratic model will change. 

\begin{remark}
\begin{sloppypar}
Some derivative-free methods are function value free, e.g., the covariance matrix adaptation evolution strategy (CMA-ES) \cite{CMA-ES}, Nelder-Mead method \cite{Nelder1965} (only using the comparison evaluation), and the algorithm using only Boolean-valued function comparisons \cite{jamieson2012query}. This article focuses on modifying the model-based algorithm relying on the function values to solve the problems with transformed objective functions and its property. We will provide a new point of view to understand the influence of the transformation on the objective functions when applying function value dependent derivative-free optimization methods.
\end{sloppypar}
\end{remark}

\subsection{Contributions and organization}

Considering solving the proposed DFOTO problems with derivative-free trust-region algorithms based on the least Frobenius norm updating quadratic model, this article contains the following contributions. 
\begin{enumerate}
\item[(a)] We modify Powell's updating formula \cite{Leastf}  of the least Frobenius norm updating model for model-based trust-region methods (by modifying an $n$-dimensional vector) with the proposed query scheme. 

\item[(b)] We analyze the least Frobenius norm updating interpolation model when minimizing transformed objective functions. We propose optimality-preserving transformations. We prove the necessary and sufficient condition for such transformations.  

\item[(c)] We discuss the positive monotonic transformations. We obtain the analytic expression of the least Frobenius norm updating model of an affinely transformed objective function. We analyze its interpolation error and give a further discussion on it.   

\item[(d)] We give a primary convergence analysis for first-order critical points. 

\item[(e)] {We present numerical results. The numerical results indicate the necessity of using the modified model updating formula. The numerical results support that our method can efficiently and robustly solve most test problems and a presented real-world problem with the transformations, even though such transformations will change the optimality of the model function.}
\end{enumerate} 

To our best knowledge, this is the first work minimizing transformed objective functions with model-based algorithms that need to use the function value information (not the comparison of the function values). The corresponding methods perform well when solving order-preserving or optimality-preserving transformed problems.

\begin{sloppypar}
{\bf Organization:} This article is organized as follows. We give our algorithm and query scheme in Section \ref{Algorithm, Query Scheme and Optimality-preserving Transformation}, which contains the least Frobenius norm updating quadratic models, trust-region subproblem, and definitions related to optimality-preserving transformations. The existence of the model optimality-preserving transformations, beyond translation transformation, is proved. The necessary and sufficient condition for such transformations is also given in this section. In Section \ref{Positive Monotonic Transformations and Affine Transformations}, we present the property of positive monotonic transformations, especially the affine transformation. We discover that the important affine  transformation with a (non-trivial) positive multiplication coefficient is not model optimality-preserving. The corresponding transformation on the least Frobenius norm updating model functions is provided, when the objective function is affinely  transformed. Section \ref{Convergence} presents the coefficients of the fully linear model's interpolation error when there is an affine transformation. Section \ref{Convergence} also analyzes the convergence of the solver for minimizing corresponding transformed objective functions under the guarantee of a provable model-based derivative-free framework. Section \ref{Numerical Results} presents numerical results of minimizing a test example, and the performance profile and sensitivity profile of solving a set of test problems with randomly affinely transformed objective functions. Such numerical results support our theoretical analysis, and the results of solving the private engineering optimal design problem of the space traveling wave tube also show the practical advantage of our method. At the end of the article, we propose the ``moving-target'' DFO problem and give conclusions of our work. 
\end{sloppypar}

\section[Algorithm, query scheme \& transformation]{Algorithm, query scheme and optimality-preserving transformation}
\label{Algorithm, Query Scheme and Optimality-preserving Transformation}

We present our algorithms and query scheme, and introduce details of the least Frobenius norm updating quadratic model for derivative-free optimization. Besides, we will introduce the trust-region subproblem in model-based derivative-free algorithms. Some basic concepts will be given, which include the optimality-preserving transformation.

\subsection{Model-based trust-region algorithms and query scheme}

We give some details of model-based trust-region algorithms and the query scheme, which will be used to solve problem (\ref{transformproblem}). We will also explain why we choose such a framework, interpolation model functions and the query scheme.

A basic framework of model-based derivative-free trust-region algorithms (for the transformed objective functions) is shown in Algorithm \ref{algo-TR}, where some details are omitted from this pseudocode for simplicity. A provable algorithmic framework of model-based derivative-free trust-region algorithms for original objective functions can be found in the book of Conn, Scheinberg and Vicente \cite{conn2009introduction}. {As shown in Assumption \ref{Qassumption} and Table \ref{table1}, the query in Algorithm \ref{algo-TR} is done for a batch of points that share the uniform transformation. The functions \(f_1, \cdots , f_k, \cdots\) separately denote the transformed objective functions corresponding to transformations \(T_1, \cdots , T_k, \cdots\) at the steps \(1, \cdots , k, \cdots\). When solving the trust-region subproblem, the center of the trust region is usually set as $\boldsymbol{x}_{\text{opt}}$, which is the interpolation point with the minimum function value among the interpolation points at the current iteration. We use $\boldsymbol{x}^{(k)}$ in the algorithmic framework for simplicity, and it will be written as $\boldsymbol{x}_{\text{opt}}$ in (\ref{trust_sub_trial_step}). When updating the \(k\)-th interpolation set, we typically drop the worst point $\boldsymbol{x}_t$ at this step and replace it by $\boldsymbol{x}_{\text{new}}(=\boldsymbol{x}^{(k-1)}+\dd_{k-1})$, which is the newly added interpolation point. The poisedness is used to measure how well the distribution of the interpolation points of the set is, in the sense of giving good quadratic interpolation models. Algorithm \ref{algo-TR} uses the traditional common \(\Lambda\)-poisedness used by the model-based trust-region derivative-free methods, and the check follows the same step in the traditional algorithm framework of the model-based trust-region derivative-free methods \cite{conn2009introduction}.}

\begin{algorithm}
        \caption{Framework of model-based derivative-free trust-region algorithms for minimizing transformed objective functions\label{algo-TR}} 
      \begin{algorithmic}[1]
        \State Given an initial point $\boldsymbol{x}^{(1)}$ and an initial set of interpolation points $\mathcal{X}_1$ satisfying $\boldsymbol{x}^{(1)} \in \mathcal{X}_1$, and
\[
f_1(\boldsymbol{x}^{(1)})=\min_{\boldsymbol{y} \in \mathcal{X}_1}\ f_1(\boldsymbol{y}).
\]
Choose the initial trust-region radius $\Delta_1$. $k:=1$.
\State \label{step2ofalg} Construct a quadratic model $Q_k$ that satisfies the interpolation condition 
\(
Q_k(\boldsymbol{x})=f_k(\boldsymbol{x}),\ \boldsymbol{x} \in \mathcal{X}_k.
\) 
\State 
Solve the trust-region subproblem 
\[
\begin{aligned}
\min_{\dd \in {\mathbb{R}}^n} \ &Q_k (\boldsymbol{x}^{(k)}+\dd) \\ 
\text{subject to } &\|\dd\|_2 \leqslant \Delta_k,
\end{aligned}
\]
and obtain the trial step $\dd_k$.
\State 
If $\boldsymbol{x}^{(k)}+\dd_k$ is accepted, e.g., $f_{k+1}(\boldsymbol{x}^{(k)}+\dd_k)<f_{k+1}(\boldsymbol{x}^{(k)})$, then $\boldsymbol{x}^{(k+1)}:=\boldsymbol{x}^{(k)}+\dd_k$; otherwise $\boldsymbol{x}^{(k+1)}:=\boldsymbol{x}^{(k)}$.
\State Check if the interpolation set is well-poised. If necessary, perform a model-improvement step to improve the poisedness of the interpolation set. Update the interpolation set to $\mathcal{X}_{k+1}$ so that $\mathcal{X}_{k+1}$ contains $\boldsymbol{x}^{(k+1)}$.
\State Update the trust-region radius based on the performance of $\dd_k$ and the poisedness of the interpolation set to obtain $\Delta_{k+1}$. $k:=k+1$. Go to Step \ref{step2ofalg}.
\end{algorithmic}
\end{algorithm}

Similar to trust-region methods using derivatives of the objective function, a derivative-free one constructs a model function of the objective function at each iteration and then minimizes the model within a trust region to obtain a new trial point. Since the derivative information can not be used, the model here is constructed by quadratic interpolation. Section \ref{Least Frobenius norm updating quadratic models for transformed objective functions} will describe how to obtain such a model. More details of derivative-free trust-region methods can be seen in the survey of Larson, Menickelly and Wild \cite{larson2019derivative}, the book of Conn, Scheinberg and Vicente \cite{conn2009introduction}, and the book of Audet and Hare \cite{audet2017derivative}.

We present the query and evaluation process for solving problems with transformed objective functions here, which will be used in the rest discussion. We synchronize the query of the first $m$ interpolation points. Once a new iteration point is obtained, the interpolation set\footnote{We call the set of interpolation points ``the interpolation set''.} is updated at each step according to the process shown in Table \ref{table1}, and we obtain the function evaluations at the points in the new interpolation set. Notice that each query set contains \(m\) points. The query set can be updated in different ways, e.g., $\mathcal{X}_k:=\mathcal{X}_{k-1}\backslash\{\arg\underset{\boldsymbol{y}\in\mathcal{X}_{k-1}}{\max}\ \Vert \boldsymbol{y}-\boldsymbol{x}^{(k)}\Vert_2\}\cup\{\boldsymbol{x}^{(k)}\}$ or $\mathcal{X}_k:=\mathcal{X}_{k-1}\backslash\{\arg\underset{\boldsymbol{y}\in\mathcal{X}_{k-1}}{\max}\ f(\boldsymbol{y})\}\cup\{\boldsymbol{x}^{(k)}\}$. Notice that \(\boldsymbol{x}^{(k)}\) here refers to the new obtained point (usually \(\boldsymbol{x}^{(k-1)}+\boldsymbol{d}_{k-1}\)) before the acceptance. In this way, it is ensured that the transformed function values corresponding to each interpolation set are with the uniform transformation.

\begin{table}[htbp] 
  \centering   
  \caption{Query and evaluation in algorithms  for solving problems with transformed objective functions\label{table1}}  
  \begin{tabular}{ccc}  
\toprule
    Step  & Query set  & {Function evaluation set} 
\\
 \midrule 
    $1$ & $\mathcal{X}_1$  & $\{f_1(\boldsymbol{x}),\ \forall \  \boldsymbol{x}\in\mathcal{X}_1\}$\\
    $2$ & $\mathcal{X}_2$  & $\{f_2(\boldsymbol{x}),\ \forall \   \boldsymbol{x}\in\mathcal{X}_2\}$\\
    \vdots &\vdots&\vdots\\
    $k$ & $\mathcal{X}_k$  & $\{f_{k}(\boldsymbol{x}),\ \forall  \  \boldsymbol{x}\in\mathcal{X}_{k}\}$\\
      \vdots &\vdots&\vdots\\ 
\botrule
  \end{tabular}
\end{table}

Notice that the interpolation model function used in Algorithm \ref{algo-TR} for each interpolation point is exactly the least Frobenius norm updating quadratic model. In the following, we give the reasons for choosing the query scheme shown in Table \ref{table1} when using Algorithm \ref{algo-TR}. First, one can find in Section \ref{Least Frobenius norm updating quadratic models for transformed objective functions} that when we use Algorithm \ref{algo-TR} based on the least Frobenius norm updating quadratic model to solve DFOTO problems, we only need to change the vector on the right-hand side of the interpolation equation (details are in (\ref{form_vec_r})), and it is simple to handle the transformation. This makes our method simple and robust in the implementation. Second, although the transformation of the output of function values will change at different iteration steps, it is reasonable to trust the points that have been evaluated in the previous iteration and rely on them to find the next iteration points with the new function value information. Otherwise, our method would suffer from iterative discontinuity if the algorithm completely re-searches all the detection points. In fact, our theoretical analysis and numerical results will support the above exposition.

There are mainly three reasons why we consider using the under-determined quadratic interpolation function\footnote{The number of interpolation points is smaller than the number of elements in the polynomial basis.} (more specifically, the least Frobenius norm updating quadratic model) in the trust-region framework. First, the quadratic model function can locally capture the curvature information of the objective function. Second, using fewer interpolation points can reduce the number of function evaluations. Last but not least, the sample/interpolation points would have to be reasonably close to the current iteration. However, a fully quadratic model needs $\mathcal{O}(n^2)$ points, and there will typically occur situations in which the number of useful sample points is lower or much lower than the number of elements in the polynomial basis.

\subsection[The least Frobenius norm updating quadratic models]{The least Frobenius norm updating quadratic models for transformed objective functions}
\label{Least Frobenius norm updating quadratic models for transformed objective functions}

In the model-based derivative-free trust-region method, $Q_k$ denotes the $k$-th quadratic model function, and $\mathcal{X}_k$ denotes the interpolation set at the $k$-th iteration. The number of interpolation points, denoted by $m$, satisfies that $n+2 \leqslant m \leqslant \frac{1}{2}(n+1)(n+2)$, which is flexible in the interpolation. Notice that the coefficients of the quadratic\footnote{The degree is not larger than 2.} model function $Q_k$ are the symmetric Hessian matrix $\nabla^2Q_k$, the gradient vector $\nabla Q_k$ and the constant term, whose freedom is $\frac{1}{2}(n+1)(n+2)$ together, which is $\mathcal{O}(n^2)$. When the problem's dimension $n$ is large, if we directly determine the coefficients of the model function $Q_k$ by solving interpolation equations $Q_k(\boldsymbol{x}_i)=f(\boldsymbol{x}_i), \ i=1,\cdots,\frac{1}{2}(n+1)(n+2)$, the number of function evaluations is large. To reduce the number of calling functions, Powell \cite{powell2006newuoa} advised that we can use fewer interpolation points to obtain the quadratic model function. To uniquely determine the coefficients of $Q_k$, the iterative quadratic model $Q_k$ is designed to be the solution of
\[
\begin{aligned}
\underset{Q}{\operatorname{\min}}\ &\left\Vert\nabla^{2} Q-\nabla^{2} Q_{k-1}\right\Vert_{F}^{2} \\  
\text{subject to} \ &Q(\boldsymbol{x})=f(\boldsymbol{x}), \ \forall \ \boldsymbol{x} \in \mathcal{X}_{k},
\end{aligned}
\]
where $\Vert\cdot\Vert_F$ denotes the Frobenius norm, i.e., \(
\Vert \boldsymbol{G}\Vert_F^2=\sum_{i=1}^{m}\sum_{j=1}^{n}\vert \boldsymbol{G}_{ij}\vert^2,
\) for given $\boldsymbol{G}\in{\mathbb{R}}^{m\times n}$.      
In other words, the quadratic model function $Q_k$ satisfies that $\nabla^{2} Q_k-\nabla^{2} Q_{k-1}$ has the least Frobenius norm over the quadratic functions satisfying the interpolation conditions $Q(\boldsymbol{x}_i)=f(\boldsymbol{x}_i), \ i=1, \cdots, m${, where $\boldsymbol{x}_1,\cdots,\boldsymbol{x}_m$ denote the current interpolation points and \(n+2 \leqslant m<\frac{1}{2}(n+1)(n+2)\).}   
\begin{remark}
For the simplicity and clearness of the discussion, we denote the interpolation points at the step being discussed as $\boldsymbol{x}_1,\cdots,\boldsymbol{x}_m$. The lower bound of \(m\) is set to be \(n+1\) for including the discussion about the case where \(\nabla^2 Q_{k-1}\) is not a zero matrix.
\end{remark}

According to the convexity of the Frobenius norm, it is proved that the quadratic model $Q_k$ can be obtained uniquely at the $k$-th iteration. Thus the rest of the coefficients' freedom is taken up. An advantage of such the model is that it has the projection property {for quadratic functions} \cite{Leastf}, and then
 \[
\Vert \nabla^2 Q_k - \nabla^2 f \Vert_F \leqslant \Vert \nabla^2 Q_{k-1} - \nabla^2 f \Vert_F  
 \]
holds for any {quadratic function} \(f\).  
 
Notice that the corresponding objective function changes with the iteration, which is exactly an important characteristic of the step-transformed problems. Notice that the term step-transformed means that the objective functions' transformations depend on the iteration/query step. Thus we come up with the idea to analyze the least Frobenius norm updating quadratic model function to adapt to the derivative-free optimization with transformed objective functions. The new updating formula is supposed to hold the ability to work when the transformed objective function $f_k$ changes with the iteration number $k$. We should notice that the model analyzed in this article is the least Frobenius norm updating quadratic model, not the least Frobenius norm one \cite{conn1997convergence,conn1997recent,conn1998derivative}, and the latter one will have different and simpler results than that in Section \ref{Positive Monotonic Transformations and Affine Transformations}. Efficient and robust numerical performance with the projection property and the yearning for more theoretical analysis of the least norm updating type model motivate us to find more details of such a model, especially with transformations existing.

For simplicity, suppose that the poised interpolation set at the $k$-th iteration is $\mathcal{X}_k=\{\boldsymbol{x}_1,\cdots,\boldsymbol{x}_m\}$ in the following. The quadratic model $Q_k$ of the transformed function $f_k$ is obtained by solving
\begin{equation}\label{leastfrob-2}
\begin{aligned}
\underset{Q}{\operatorname{\min}}\ &\left\Vert \nabla^{2} Q-\nabla^{2} Q_{k-1}\right\Vert_{F}^{2} \\  
\text{subject to}  \ &Q(\boldsymbol{x})=f_k(\boldsymbol{x}), \ \forall \ \boldsymbol{x} \in \mathcal{X}_{k}.
\end{aligned}
\end{equation}
We define $D_k(\boldsymbol{x})=Q_{k}(\boldsymbol{x})-Q_{k-1}(\boldsymbol{x})$. Then, from (\ref{leastfrob-2}), $D_k(\boldsymbol{x})$ can be obtained by solving
\begin{equation}
\label{miniFrob}
\begin{aligned}
\underset{D}{\operatorname{\min}} \  &\left\Vert \nabla^2 D\right\Vert_F^2 \\ 
\text{subject to} \
&\left\{
\begin{aligned} 
&
D(\boldsymbol{x}_{i})=f_{k}(\boldsymbol{x}_i)-f_{k-1}(\boldsymbol{x}_i), \ i=1,\cdots, t-1, t+1, \cdots, m,\\
&D(\boldsymbol{x}_{\text{new}})=f_{k}(\boldsymbol{x}_{\text{new}})-Q_{k-1}(\boldsymbol{x}_{\text{new}}),
\end{aligned}
\right.
\end{aligned}
\end{equation}
since the old $\boldsymbol{x}_t$ is dropped and replaced by $\boldsymbol{x}_{\text{new}}$ at the current ($k$-th) iteration, according to the framework of the model-based derivative-free trust-region algorithms \cite{conn2009introduction}. {The derivation from (\ref{leastfrob-2}) to (\ref{miniFrob}) is straightforward, and the equivalence holds after replacing the function \(Q(\boldsymbol{x})-Q_{k-1}(\boldsymbol{x})\) by \(D(\boldsymbol{x})\).}

Let $\lambda_j, j=1,2,\cdots,m$, be the Lagrange multipliers for the KKT conditions of optimization problem (\ref{miniFrob}), which, as Powell \cite{Leastf} pointed out, have the properties that
\begin{equation} 
\label{new-*}
\left\{
\begin{aligned}
&\sum_{j=1}^{m} \lambda_{j}=0, \  \sum_{j=1}^{m} \lambda_{j}\left(\boldsymbol{x}_{j}-\boldsymbol{x}_{0}\right)=\boldsymbol{0},\\
&\nabla^{2} D_k=\sum_{j=1}^{m} \lambda_{j}\left(\boldsymbol{x}_{j}-\boldsymbol{x}_{0}\right)\left(\boldsymbol{x}_{j}-\boldsymbol{x}_{0}\right)^{\top},
\end{aligned}
\right.
\end{equation}
where $\boldsymbol{0}\in{\mathbb{R}}^{m}$, $\boldsymbol{x}_0$ is the base point to reduce the computation errors, and it is initially chosen as the input start point. 
Therefore, the quadratic function $D_k(\boldsymbol{x})$ can be written in the form 
\begin{equation} 
\label{D_k_form}
D_k(\boldsymbol{x})=c+(\boldsymbol{x}-\boldsymbol{x}_{0})^{\top} \g+\frac{1}{2} \sum_{j=1}^{m} \lambda_{j}\left((\boldsymbol{x}-\boldsymbol{x}_{0})^{\top}(\boldsymbol{x}_{j}-\boldsymbol{x}_{0})\right)^{2},\ \boldsymbol{x}\in{\mathbb{R}}^{n}. 
\end{equation} 
After determining the parameters ${\boldsymbol{\lambda}}^{\top}=(\lambda_1,\cdots,\lambda_m)^{\top} \in {\mathbb{R}}^m$, $c \in {\mathbb{R}}$ and $\g \in {\mathbb{R}}^{n}$, we can determine the unique function $D_k(\boldsymbol{x})$ and thus obtain the new quadratic model function $Q_{k}(\boldsymbol{x})$. It is easy to see that the coefficients of $D(\boldsymbol{x})$ are the solution of the system of linear equations
\[
\left(\begin{array}{cc}
\boldsymbol{A} & \boldsymbol{X}^{\top} \\
\boldsymbol{X} & \boldsymbol{0}
\end{array}\right)
\left(
{\boldsymbol{\lambda}}^{\top},
c,
\g^{\top}
\right)^{\top}=\left(
\rr^{\top}, 
0,\cdots,0
\right)^{\top},
\]
where the matrix $\boldsymbol{0}\in{\mathbb{R}}^{(n+1)\times(n+1)}$. The elements of the matrix $\boldsymbol{A}\in{\mathbb{R}}^{m\times m}$ and $\boldsymbol{X}\in{\mathbb{R}}^{(n+1)\times m}$ are
\[
\boldsymbol{A}_{i j} =\frac{1}{2}\left(\left(\boldsymbol{x}_{i}-\boldsymbol{x}_{0}\right)^{\top}\left(\boldsymbol{x}_{j}-\boldsymbol{x}_{0}\right)\right)^{2}
\] 
and
\[  
\boldsymbol{X} =\left(\begin{array}{ccc}
1 & \cdots & 1 \\
\boldsymbol{x}_{1}-\boldsymbol{x}_{0} & \cdots  & \boldsymbol{x}_{m}-\boldsymbol{x}_{0}
\end{array}\right),
\]
where $1\leqslant i, j \leqslant m$. Besides, the vector $\rr\in{\mathbb{R}}^{m}$ has the form as
\begin{equation}
\label{form_vec_r}
\begin{aligned}
\rr=&\left(
f_{k}(\boldsymbol{x}_1)-f_{k-1}(\boldsymbol{x}_1),
\cdots,
f_{k}(\boldsymbol{x}_{t-1})-f_{k-1}(\boldsymbol{x}_{t-1}),
f_{k}(\boldsymbol{x}_{\text{new}})-Q_{k-1}(\boldsymbol{x}_{\text{new}}),\right.\\
&\left.f_{k}(\boldsymbol{x}_{t+1})-f_{k-1}(\boldsymbol{x}_{t+1}),
\cdots,
f_{k}(\boldsymbol{x}_m)-f_{k-1}(\boldsymbol{x}_m)
\right)^{\top}.
\end{aligned}
\end{equation}

\begin{remark}
The vector $\rr$ is the only difference between the updating formula for minimizing objective functions with transformations and the one for minimizing objective functions without transformations. This form is natural but of major importance for obtaining the least Frobenius norm updating quadratic models of transformed objective functions.
\end{remark}

We denote $\boldsymbol{W}$ as the KKT matrix
\begin{equation*} 
\boldsymbol{W}=\left(\begin{array}{cc}
\boldsymbol{A} & \boldsymbol{X}^{\top} \\
\boldsymbol{X} & \boldsymbol{0}
\end{array}\right). 
\end{equation*}  
Thus, if \(\boldsymbol{W}\) is invertible, \({\boldsymbol{\lambda}},c\) and \(\g\) can be obtained by 
\begin{equation}
\label{lambda-c-g-obtain}
\left(
{\boldsymbol{\lambda}}^{\top},
c,
\g^{\top}
\right)^{\top}
=\boldsymbol{H}
\left(
\rr,
0,
\cdots,
0
\right)^{\top},
\end{equation}
where \(\boldsymbol{H}=\boldsymbol{W}^{-1}\). {The invertibility of \(\boldsymbol{W}\) depends on the interpolation points' positions and refers to the set's poisedness, and it has been deeply discussed by Powell \cite{Powell04onupdating}. The initial interpolation points guarantee the invertibility of the initial \(\boldsymbol{W}\), and the invertibility of the iterative \(\boldsymbol{W}\) and the numerical exactness of the formula (\ref{lambda-c-g-obtain}) will be iteratively guaranteed by selecting suitable interpolation points in the model-improvement steps. This part is the same as the discussion and method in Powell's work \cite{Powell04onupdating,powell2006newuoa}. We assume that the matrix \(\boldsymbol{W}\) is invertible in the following discussion.}  

\begin{remark}
If we obtain the $k$-th model function by solving problem
\begin{equation}
\label{leasthessianandgradient}
\begin{aligned}
\underset{Q}{\operatorname{\min}}\ &\left\Vert \nabla^{2} Q-\nabla^{2} Q_{k-1}\right\Vert_{F}^{2}+\sigma\left\Vert \nabla Q-\nabla Q_{k-1}\right\Vert_{2}^{2} \\ 
\text{subject to}  \ &Q(\boldsymbol{x})=f_k(\boldsymbol{x}), \ \forall \ \boldsymbol{x} \in \mathcal{X}_{k},
\end{aligned}
\end{equation}
where the weight coefficient $\sigma\geqslant 0$, then the following results and analysis still hold, except that $\boldsymbol{W}$ is changed to the matrix
\begin{equation}
\label{Wofleasthessianandgradient}
\boldsymbol{W}=\left(\begin{array}{cc}
\boldsymbol{A} & \boldsymbol{X}^{\top} \\
\boldsymbol{X} & \begin{array}{cc}
0 & \boldsymbol{0}^{\top} \\
\boldsymbol{0} &-\frac{\sigma}{2}\boldsymbol{I}
\end{array}
\end{array}\right),
\end{equation}
where $\boldsymbol{I}\in{\mathbb{R}}^{n\times n}$ is the identity matrix, and $\boldsymbol{0}\in{\mathbb{R}}^{n}$ is the zero vector. In fact, the conclusion of this article also holds for the least norm updating quadratic models in the sense of other norms. 
\end{remark} 

The updating formula of the matrix $\boldsymbol{H}$ given by Powell \cite{Powell04onupdating} can be used. Finally, we obtain $D_k(\boldsymbol{x})$ {as (\ref{D_k_form})} and the parameters ${\boldsymbol{\lambda}}$, $c$ and $\g$ are given by (\ref{lambda-c-g-obtain}), with $\rr$ being given by (\ref{form_vec_r}). Then we can obtain the $k$-th model \(Q_k=Q_{k-1}+D_{k}\).

\subsection{Trust-region subproblem}
\begin{sloppypar}
Model-based derivative-free trust-region algorithms compute a trial step based on solving the trust-region subproblem of the current quadratic model function, i.e., 
\begin{equation}\label{trust_sub_trial_step}
\begin{aligned}
\underset{{\dd \in {\mathbb{R}}^n}}{\operatorname{\min}} \ &Q_k(\boldsymbol{x}_{\text{opt}}+\dd)\\  
\text{subject to}  \ &\Vert \dd \Vert_{2} \leqslant \Delta_{k}.
\end{aligned}
\end{equation} 
{In (\ref{trust_sub_trial_step}), $\boldsymbol{x}_{\text{opt}}$ denotes the interpolation point with the optimal output function value in the $k$-th interpolation set $\mathcal{X}_k$.} 

Details of framework of such type of algorithms can be seen in the book of Conn, Scheinberg and Vicente \cite{conn2009introduction}.

One of the termination conditions of the subroutine for solving the quadratic model subproblem in the trust region in the corresponding model-based derivative-free trust-region algorithms, is
$
\Vert \dd_k \Vert_2 < \rho_{k},
$
where $\dd_k$ is the solution of the trust-region subproblem (\ref{trust_sub_trial_step}) and $\rho_k$ is a constant bound.
\end{sloppypar}

The parameter $\rho_{k}$ is an iterative lower bound on the trust-region radius, and it is designed to maintain enough distance between the interpolation points. 

For the derivative-free optimization problem without transformations, the objective function itself does not change, and only the trust region changes when the iteration increases. However, in DFOTO, $f_k$ and the trust region both change when the iteration increases. The quadratic model $Q_k$ is updated to approximate $f_k$. Then the solution of the subproblem, $\dd_k$, may change according to its definition. As a result, the condition $\Vert \dd_k \Vert_{2} < \rho_{k}$ may not hold at all, and consequently the termination of the algorithm will be influenced.

If the termination condition is not satisfied, the iteration of the algorithm can not leave the loop for solving the trust-region subproblem. The number of iterations may grow up, which refers to a high cost from function evaluations, without achieving the wanted approximate minimum of $f$. Besides, the model-improvement step of the methods can hardly be called, which can not effectively reduce the interpolation error of the model.

\subsection{Optimality-preserving transformations}

We give the analysis and discussion about the optimality-preserving transformations in this part.

\subsubsection{Preliminaries}

\begin{sloppypar}
Given a black-box function $f: {\mathbb{R}}^n\rightarrow {\mathbb{R}}$, we say a quadratic function $Q$ is a quadratic interpolation model of $f$ on the interpolation set $\mathcal{X}\subset {\mathbb{R}}^{n}$, if it satisfies that
\(Q(\boldsymbol{x})=f(\boldsymbol{x}),\ \forall \ \boldsymbol{x}\in \mathcal{X}\). Besides, we say an interpolation set is poised if it is poised in the least/minimum Frobenius norm sense with details in Section 5.3 in the book of Conn, Scheinberg and Vicente \cite{conn2009introduction}, which refers to the nonsingularity/invertibility of the KKT matrix. In the following, a quadratic function (or a linear function) is a function in the space of polynomials of degree less than or equal to 2 (or 1) in \({\mathbb{R}}^n\). Besides, we should note that every transformation discussed in the following is a transformation from ${\mathbb{R}}$ to ${\mathbb{R}}$. To illustrate more details, we give the following definitions.
\end{sloppypar}

\begin{definition}[The least Frobenius norm updating model of the function $h$ on $\mathcal{X}$ based on $Q_\alpha$]
Given the function $h: {\mathbb{R}}^n\rightarrow {\mathbb{R}}$, a quadratic function $Q_\alpha$ and a poised set $\mathcal{X}\subset {\mathbb{R}}^{n}$, where \(n+1\leqslant \vert \mathcal{X} \vert <\frac{1}{2}(n+1)(n+2)\), we say a quadratic model function is the least Frobenius norm updating model of $h$ on $\mathcal{X}$ based on $Q_\alpha$, if it is the solution of 
\begin{equation}
\label{thenewdefmodel}
\begin{aligned}
\underset{Q}{\operatorname{\min}}\ &\left\Vert \nabla^{2} Q-\nabla^{2} Q_{\alpha}\right\Vert_{F}^{2} \\ 
\text{subject to}  \ &Q(\boldsymbol{x})=h(\boldsymbol{x}), \ \forall \  \boldsymbol{x} \in \mathcal{X}. 
\end{aligned}
\end{equation}
{Mathematically, we denote the model above by a mapping \(\mathcal{M}^{\mathcal{X}}_{Q_{\alpha}}\), i.e., the solution of (\ref{thenewdefmodel}) is denoted by \(\mathcal{M}^{\mathcal{X}}_{Q_{\alpha}}(h)\).}   
\end{definition}

Notice that $\vert\cdot\vert$ denotes the cardinality and the solution of (\ref{thenewdefmodel}) is unique as we discussed with the analysis of Powell \cite{Leastf}. 

\begin{definition}[Subproblem of $Q$ with trust-region radius $\Delta$] 
Given a point named $\boldsymbol{x}_{\text{opt}}\in {\mathbb{R}}^{n}$\footnote{The point $\boldsymbol{x}_{\text{opt}}$ is usually set as the interpolation point with optimal function value in the interpolation set.}, a quadratic function \(Q\) and \(\Delta\in\mathbb{R}\), we call problem 
\[ 
\begin{aligned}
\underset{{\dd \in {\mathbb{R}}^n}}{\operatorname{\min}} \ &Q(\boldsymbol{x}_{\text{opt}}+\dd)\\  
\text{subject to}  \ &\Vert \dd \Vert_{2} \leqslant \Delta
\end{aligned}
\]
the subproblem of $Q$ with trust-region radius $\Delta$ centered at $\boldsymbol{x}_{\text{opt}}$. The center will be omitted in the following if there is no need to emphasize.
\end{definition}

We give the following definition of the model optimality-preserving transformation. 

\begin{sloppypar}
\begin{definition}[Model optimality-preserving transformation]
Suppose that the poised set $\mathcal{X}=\{\boldsymbol{x}_1,\cdots,\boldsymbol{x}_m\}\subset {\mathbb{R}}^{n}$, and $Q_{\alpha}$ is a quadratic function. We say a transformation $T$ is a model optimality-preserving transformation with trust-region radius $\Delta$, if the solution of the subproblem of the least Frobenius norm updating model of the function $f$ on $\mathcal{X}$ based on $Q_{\alpha}$ is the same as the solution of the subproblem of the least Frobenius norm updating model of the function $T\circ f$ on $\mathcal{X}$ based on $Q_{\alpha}$. {Mathematically, given the point \(\boldsymbol{x}_{\text{opt}}\in {\mathbb{R}}^{n}\), we say a transformation $T$ is a model optimality-preserving transformation with trust-region radius $\Delta$, if 
\[
\arg\min_{{\Vert \dd \Vert_{2} \leqslant \Delta}} \ Q_{\text{orig}}(\boldsymbol{x}_{\text{opt}}+\dd)
=\arg\min_{{\Vert \dd \Vert_{2} \leqslant \Delta}} \ Q_{\text{trans}}(\boldsymbol{x}_{\text{opt}}+\dd), 
\] 
where 
\(
Q_{\text{orig}}:=\mathcal{M}^{\mathcal{X}}_{Q_{\alpha}}(f),\ Q_{\text{trans}}:=\mathcal{M}^{\mathcal{X}}_{Q_{\alpha}}(T\circ f).
\)}
\end{definition}
\end{sloppypar}

\begin{assumption}
\label{assum-import}
Given the function $f: {\mathbb{R}}^n\rightarrow \mathbb{R}$, \(\boldsymbol{x}_{\text{opt}}\in {\mathbb{R}}^{n}\), a quadratic function \(Q_{\alpha}\) and the poised interpolation set $\mathcal{X}=\{\boldsymbol{x}_1,\cdots,\boldsymbol{x}_m\}\subset {\mathbb{R}}^{n}$, where $n+1 \leqslant \vert \mathcal{X}\vert=m<\frac{1}{2}(n+1)(n+2)$, we assume that $\dd^*\in {\mathbb{R}}^n$ is the solution of the subproblem of the least Frobenius norm updating model of the function $f$ on $\mathcal{X}$ based on the quadratic function $Q_\alpha$ with trust-region radius $\Delta$, where $\Vert \dd^*\Vert_2<\Delta$, and such a model is strictly convex.
\end{assumption}

\subsubsection{Theoretical results}

\begin{sloppypar} 
We give some theoretical results in this part. We obtain the following necessary and sufficient condition for a model optimality-preserving transformation.
\end{sloppypar}

\begin{theorem}
\label{coro-pc-2}
\begin{sloppypar} 
Suppose that Assumption \ref{assum-import} holds. Then the transformation $T$ is a model optimality-preserving transformation if and only if $({T}(f(\boldsymbol{x}_1)),\cdots,{T}(f(\boldsymbol{x}_m)))^{\top}$ is the solution of linear equations 
\begin{equation}\label{1-1}
\begin{aligned}
&\sum_{j=1}^{m} \left(\left(\boldsymbol{x}_{j}-\boldsymbol{x}_{\rm opt}\right)\left(\boldsymbol{x}_{j}-\boldsymbol{x}_{\rm opt}\right)^{\top}\dd^*\right) \bigg(\boldsymbol{H}_j\left({T}(f(\boldsymbol{x}_1)) - {Q}_{\alpha}(\boldsymbol{x}_{1}),\cdots, {T}(f(\boldsymbol{x}_m)) \right.\\
&\left.- {Q}_{\alpha}(\boldsymbol{x}_{m}), 0,\cdots, 0\right)^{\top}\bigg)+\nabla^2{Q}_{\alpha}\dd^*=-\left(\boldsymbol{H}^{\top}_{m+2},\cdots,\boldsymbol{H}^{\top}_{m+n+1}\right)^{\top}\bigg({T}(f(\boldsymbol{x}_1)) \\
&- {Q}_{\alpha}(\boldsymbol{x}_{1}),\cdots,{T}(f(\boldsymbol{x}_m)) - {Q}_{\alpha}(\boldsymbol{x}_{m}),0,\cdots, 0\bigg)^{\top}-\nabla Q_{\alpha}(\boldsymbol{x}_{\rm opt}),
\end{aligned}
\end{equation}
where $\boldsymbol{H}$ is the inverse matrix of the KKT matrix, and $\boldsymbol{H}_{j}$ denotes the $j$-th row of  $\boldsymbol{H}$ here.  
\end{sloppypar}
\end{theorem}

\begin{proof}
Suppose that the least Frobenius norm updating model of the function $T\circ f$ on $\mathcal{X}$ based on $Q_\alpha$ is 
\[
Q_{u}(\boldsymbol{x})=Q_{\alpha}(\boldsymbol{x})+{c}_{u}+(\boldsymbol{x}-\boldsymbol{x}_{\text{opt}})^{\top}{\g}_{u} +\frac{1}{2} \sum_{j=1}^{m} ({{\boldsymbol{\lambda}}_{u}})_{j}\left(\left(\boldsymbol{x}-\boldsymbol{x}_{\text{opt}}\right)^{\top}\left(\boldsymbol{x}_{j}-\boldsymbol{x}_{\text{opt}}\right)\right)^2.
\]
We can obtain that 
\[
\left(
{{\boldsymbol{\lambda}}}_u^{\top},
{c}_u,
{\g}_u^{\top}
\right)^{\top}=\boldsymbol{H}\left(
{T}(f(\boldsymbol{x}_1)) - {Q}_{\alpha}(\boldsymbol{x}_{1}),
\cdots,
{T}(f(\boldsymbol{x}_m))  - {Q}_{\alpha}(\boldsymbol{x}_{m}),
0,
\cdots,
0
\right)^{\top},
\]
and the base point $\boldsymbol{x}_0$ appearing in the KKT matrix is set as $\boldsymbol{x}_{\text{opt}}$ for simplicity.

According to the KKT conditions \cite{NoceWrig06}, $\dd^*$ is feasible, and it satisfies that
\(
\nabla^{2} {Q}_{u} \dd^{*}=-{\g}_u,
\)  
since $\Vert \dd^*\Vert_2<\Delta$. Combining with (\ref{new-*}) and the definition of $\boldsymbol{H}$, we thus obtain the necessary and sufficient condition.
\end{proof}

\begin{remark}
\label{more_other_translation} 
Suppose that Assumption \ref{assum-import} holds. It holds that for any \(C_2\in \mathbb{R}\), $(f(\boldsymbol{x}_1)+C_2,\cdots,f(\boldsymbol{x}_m)+C_2)^{\top}$ is a solution of linear equations (\ref{1-1}), and hence the translation transformation $T$, satisfying $T\circ f=f+C_2$, is a model optimality-preserving transformation. The conclusion is consistent with Corollary \ref{full}. Besides, if \(\vert \mathcal{X}\vert\geqslant  n+2\), there are more other model optimality-preserving transformations according to Theorem \ref{coro-pc-2}.
\end{remark}

The solution space of (\ref{1-1}) contains a translated linear space, of which the dimension is at least $m-n$. Powell requests \(m\geqslant n+2\) for the least Frobenius norm updating quadratic model applied in NEWUOA \cite{powell2006newuoa}. For Remark \ref{more_other_translation}, if \(n+2 \leqslant m<\frac{1}{2}(n+1)(n+2)\), the model optimality-preserving transformation can be other transformations in addition to the transformations satisfying 
\begin{equation}
\label{onlytranslation}
(T(f(\boldsymbol{x}_1)),\cdots,T(f(\boldsymbol{x}_m)))^{\top}=({f(\boldsymbol{x}_1)}+C_{2},\cdots,{f(\boldsymbol{x}_m)}+C_{2})^{\top}
\end{equation} 
for \(C_2\in \mathbb{R}\). Actually, to obtain the fully linear property, which is an important property of under-determined quadratic model functions, it is shown in the book of Conn, Scheinberg and Vicente \cite{conn2009introduction} that there should be at least \(n+1\) interpolation points. In the following, we present a natural assumption before the further discussion.

\begin{sloppypar}
\begin{assumption}
\label{linearly-independent}
Suppose that the homogeneous linear equations of $(T(f(\boldsymbol{x}_1)),\cdots,$ $T(f(\boldsymbol{x}_m)))^{\top}$ referring to (\ref{1-1}) are linearly independent. 
\end{assumption}
\end{sloppypar}
We then obtain the following corollary. 
\begin{corollary} 
\label{corollaryaboutonlysmallm}
Suppose that Assumptions \ref{assum-import} and \ref{linearly-independent} hold. If \(m=n+1\), then a transformation $T$ is a model optimality-preserving transformation if and only if it satisfies (\ref{onlytranslation}). 
\end{corollary}

\begin{proof}
When Assumption \ref{linearly-independent} holds, if \(m=n+1\), the dimension of the solution space of (\ref{1-1}) is 1. Thus we arrive at the conclusion. 
\end{proof}

\begin{remark}
If the subproblem for obtaining the quadratic model function is chosen as (\ref{leasthessianandgradient}), then the results of Theorem \ref{coro-pc-2}, Corollary \ref{corollaryaboutonlysmallm} and the analysis above still holds for the corresponding matrix \(\boldsymbol{H}\) being the inverse matrix of the KKT matrix in (\ref{Wofleasthessianandgradient}). Moreover, in the case where \(m \leqslant n\) and the rest parts of Assumptions \ref{assum-import} and \ref{linearly-independent} hold, a transformation $T$ is a model optimality-preserving transformation if and only if it satisfies
\(
(T(f(\boldsymbol{x}_1)),\cdots,T(f(\boldsymbol{x}_m)))^{\top}=({f(\boldsymbol{x}_1)},\cdots,{f(\boldsymbol{x}_m)})^{\top}.
\) The conclusion above holds directly since the solution of (\ref{1-1}) is unique if \(m \leqslant n\).
\end{remark}

\subsubsection{An example}

An example of model optimality-preserving transformation is in the following. 
\begin{sloppypar}
\begin{example}
\label{realexample}
Suppose that the base point $\boldsymbol{x}_0$ and initial interpolation points $\boldsymbol{x}_1,\boldsymbol{x}_2,\boldsymbol{x}_3,\boldsymbol{x}_4,\boldsymbol{x}_5$ are $\boldsymbol{x}_0=(0,0)^{\top}, \boldsymbol{x}_1=(0,0)^{\top},\boldsymbol{x}_2=(1,0)^{\top},\boldsymbol{x}_3=(0,1)^{\top},\boldsymbol{x}_4=(-1,0)^{\top},\boldsymbol{x}_5=(0,-1)^{\top}$, and the original black-box objective function is 
\[
f(x,y)=\frac{1}{2}((x-y)^2+(x-1)^2+(y-1)^2).
\]
Notice that $x$ and $y$ here denote the components of the 2-dimensional variable. In this example, we set the trust-region radius as 10. We know that
\(f(\boldsymbol{x}_1)=1,f(\boldsymbol{x}_2)=1,f(\boldsymbol{x}_3)=1,f(\boldsymbol{x}_4)=3,f(\boldsymbol{x}_5)=3\), $\boldsymbol{x}_{\text{opt}}^{(1)}=\boldsymbol{x}_{1}$, and then we can obtain that \({\boldsymbol{\lambda}}=(-4,1,1,1,1)^{\top},c=1,\g=(-1,-1)^{\top}\), and 
\[
Q_1(x,y)=1 - x - y + x^2 + y^2,
\] 
after calculating the inverse matrix of the KKT matrix, we say $\boldsymbol{H}$. Then we add the minimum point of the model in the trust region, i.e., $\boldsymbol{x}_{\text{new}}=(\frac{1}{2},\frac{1}{2})^{\top}$, into the interpolation set and drop the point $\boldsymbol{x}_{5
}$. Besides, we know that $\boldsymbol{x}_{\text{opt}}^{(2)}=\boldsymbol{x}_{\text{new}}$, and we set the base point $\boldsymbol{x}_0=\boldsymbol{x}_{\text{opt}}^{(2)}$ for simplicity. After obtaining the inverse matrix of the new KKT matrix, we say $\boldsymbol{H}_{\text{new}}$, we can know that
\(f(\boldsymbol{x}_{\text{new}})=\frac{1}{4},Q_1(\boldsymbol{x}_{\text{new}})=\frac{1}{2}\), and then we can obtain that \({\boldsymbol{\lambda}}^{+}=(\frac{2}{3},1,\frac{4}{3},-\frac{1}{3},-\frac{8}{3})^{\top}\), \(c^{+}=-\frac{1}{4},\g^{+}=(-\frac{1}{3},-\frac{1}{3})^{\top}\), \(D(x,y)=-\frac{2}{3}xy+\frac{1}{3}y^2-\frac{1}{3}y\), and 
\[
Q_2(x,y)=Q_1(x,y)+D(x,y)=x^2-\frac{2}{3}xy-x+\frac{4}{3}y^2-\frac{4}{3}y+1.
\]
Then we obtain the solution of the model function in the trust-region region $\{\boldsymbol{x}: \Vert \boldsymbol{x}-\boldsymbol{x}_{\text{opt}}^{(2)}\Vert_2\leqslant 10\}$, i.e., $\dd^{*}=(\frac{5}{22},\frac{2}{11})^{\top}$, and the next iteration/interpolation point is $(\frac{8}{11},\frac{15}{22})^{\top}$. Sustituting $\dd^*$ into equation (\ref{1-1}), we obtain the necessary and sufficient condition
\begin{equation}
\label{sufconreal}
\begin{aligned}
T(f(\boldsymbol{x}_4)) &= 2 + \frac{9 T(f(\boldsymbol{x}_1))}{10} - \frac{27 T(f(\boldsymbol{x}_2))}{5} + \frac{11 T(f(\boldsymbol{x}_3))}{2},\\
T(f(\boldsymbol{x}_{\text{new}})) &= -\frac{3}{4} + \frac{33 T(f(\boldsymbol{x}_1))}{40} + \frac{21 T(f(\boldsymbol{x}_2))}{20} - \frac{7 T(f(\boldsymbol{x}_3))}{8},
  \end{aligned}
\end{equation}
and the solution space is 
\[
\begin{pmatrix}
T(f(\boldsymbol{x}_1))\\
T(f(\boldsymbol{x}_2))\\
T(f(\boldsymbol{x}_3))\\
T(f(\boldsymbol{x}_4))\\
T(f(\boldsymbol{x}_{\text{new}}))
\end{pmatrix}=
\begin{pmatrix}
0\\
0\\
0\\
2\\
-\frac{3}{4}
\end{pmatrix}
+k_1
\begin{pmatrix}
40\\
0\\
0\\
36\\
33
\end{pmatrix}
+k_2
\begin{pmatrix}
0\\
20\\
0\\
-108\\
21
\end{pmatrix}
+k_3
\begin{pmatrix}
0\\
0\\
8\\
44\\
-7
\end{pmatrix},
\]
where $k_1,k_2,k_3\in {\mathbb{R}}$. We can see that it contains the translation transformation, which refers to the constants satisfying $k_2=2k_1, k_3=5k_1$. Notice that the original function values $f(\boldsymbol{x}_1),f(\boldsymbol{x}_2),f(\boldsymbol{x}_3),f(\boldsymbol{x}_4),f(\boldsymbol{x}_\text{new})$ satisfy (\ref{sufconreal}). 

In Fig. \ref{realexamplefig}, the one on the left-hand side contains the iteration/interpolation points at the first iteration, with original objective function values. The one on the right-hand side contains the iteration/interpolation points at the second iteration, with objective function values transformed by model optimality-preserving transformations. 

\begin{figure}[htbp]
    \centering

\fbox{
\includegraphics[width=0.95\linewidth]{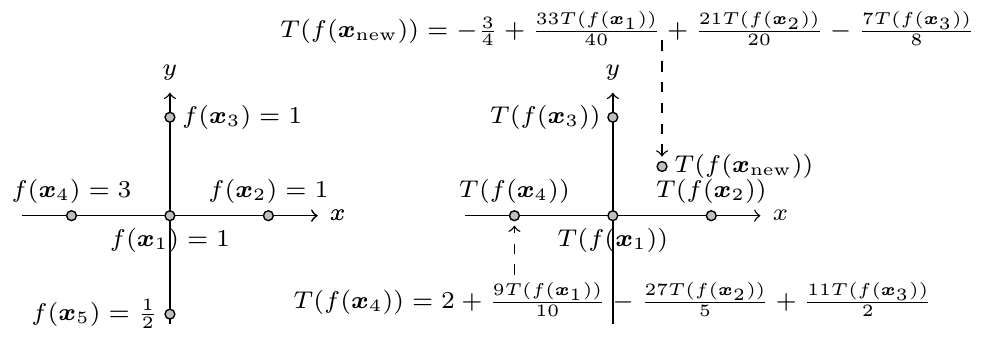}
}

    \caption{Model optimality-preserving transformations in Example \ref{realexample}\label{realexamplefig}}
\end{figure}
 
\end{example}
\end{sloppypar}

\section[Positive monotonic \& affine transformations]{Positive monotonic transformations and affine transformations}
\label{Positive Monotonic Transformations and Affine Transformations}

Notice that the solution $\boldsymbol{d}_{k}$ of the subproblem of $Q_k$ with trust-region radius $\Delta_k$ is an approximation to the solution of the subproblem of $f_k$ with trust-region radius $\Delta_k$, since $Q_{k}$ is the local quadratic interpolation model function of $f_{k}$. Therefore, we propose the definition of the objective optimality-preserving transformation in the following.  

\begin{definition}[Objective optimality-preserving transformation] 
We say a transformation $T$ is an objective optimality-preserving transformation with trust-region radius $\Delta$, if the solution of the subproblem of the objective function $f$ with trust-region radius $\Delta$ is the same as the solution of the subproblem of $T\circ f$ with trust-region radius $\Delta$. {Mathematically, given the point \(\boldsymbol{x}_{\text{opt}}\in {\mathbb{R}}^{n}\), we say a transformation $T$ is an objective optimality-preserving transformation with trust-region radius $\Delta$, if 
\[
\arg\min_{{\Vert \dd \Vert_{2} \leqslant \Delta}} \ f(\boldsymbol{x}_{\text{opt}}+\dd)
=\arg\min_{{\Vert \dd \Vert_{2} \leqslant \Delta}} \ (T\circ f)(\boldsymbol{x}_{\text{opt}}+\dd).  
\]}  
\end{definition}

This section will illustrate the objective functions with some fundamental transformations and the corresponding least Frobenius norm updating model functions. We first present the definition of positive monotonic transformations.
\begin{definition}\label{PMT}
If a transformation $T: {\mathbb{R}}\rightarrow {\mathbb{R}}$ preserves the order of data, namely $T(\theta_1)>T(\theta_2)$ for $\theta_1>\theta_2$, and $T(\theta_1)=T(\theta_2)$ for $\theta_1=\theta_2$, then we say that $T$ is a positive monotonic transformation.
\end{definition}

Then we can directly obtain the following theorem.
\begin{theorem} 
If the transformation $T$ is a positive monotonic transformation, then it is an objective optimality-preserving transformation with any trust-region radius.
\end{theorem}
\begin{proof}
The conclusion is derived from Definition \ref{PMT}. 
\end{proof}

Positive monotonic transformations can be any strictly monotonic increasing functions, such as linear functions with positive (multiplication) coefficients, exponential functions, and power functions with odd positive power. 
Here, we give the simplest example, namely the affine transformation.
\begin{example}
An affine transformation $T$, satisfying $T\circ f=C_{1}f+C_{2}$, where the constants $C_{1}, C_{2}\in {\mathbb{R}}$, and $C_{1}>0$, is a positive monotonic transformation.
\end{example}

We will see that the least Frobenius norm updating quadratic model can not be obtained by the same transformation as that of the original objective function, even if the objective function $f$ is affinely transformed to be $C_{1} f+C_{2}$ at a step, where $C_{1}>0$. In other words, in the previous section, we have shown that the affine transformation is normally not a model optimality-preserving transformation. However, the case is fundamental and of practical importance where the objective function is affinely transformed before the output. Therefore, in the rest of this article, we will further discuss the model of affinely transformed objective functions, the analytic expression of its model function, and the fully linear constants of the interpolation model. Moreover, we will give the corresponding numerical experiments and attempt to use them as typical examples of objective optimality-preserving transformations to show the numerical efficiency of our method.

We give the following theorem to obtain the exact expression of the corresponding model of the affinely transformed objective functions. In the rest of this section, we use the notation $\boldsymbol{x}_t$ to denote $\boldsymbol{x}_{\text{new}}$ for simplicity, since $\boldsymbol{x}_{\text{new}}$ has indeed been put at the $t$-th position of the interpolation set before obtaining the $k$-th model. 
\begin{theorem}\label{C1,k+C2,k}
Suppose that $Q_{\alpha}$ is a quadratic function, and $\mathcal{X}=\{\boldsymbol{x}_1,\cdots,\boldsymbol{x}_m\}\subset {\mathbb{R}}^{n}$ is a poised set, where \(n+1 \leqslant m<\frac{1}{2}(n+1)(n+2)\). Then it holds that 
\begin{equation}
\label{equationoftransformationC1f+C2}
\mathcal{M}^{\mathcal{X}}_{Q_{\alpha}}(C_1f+C_2)=(C_1\mathcal{M}^{\mathcal{X}}_{Q_{\alpha}}(f)+C_2)+(C_{1}-1)(\mathcal{M}^{\mathcal{X}}_{0}(Q_{\alpha})-Q_{\alpha}),
\end{equation} 
for any \(C_1,C_2\in\mathbb{R}\), where $\mathcal{M}^{\mathcal{X}}_{0}(Q_{\alpha})$ denotes the least Frobenius norm updating quadratic model of the function $Q_{\alpha}$ on $\mathcal{X}$ based on the zero function, i.e., the least Frobenius norm quadratic model of \(Q_{\alpha}\). 
\end{theorem}
\begin{proof}
Let $Q_{\beta}:=\mathcal{M}^{\mathcal{X}}_{Q_{\alpha}}(f)$, $\hat{Q}_{\beta}:=\mathcal{M}^{\mathcal{X}}_{Q_{\alpha}}(C_1f+C_2)$ and $\tilde{Q}:=\mathcal{M}^{\mathcal{X}}_{0}(Q_\alpha)$.   
Let $D_{\beta}:=Q_{\beta}-Q_{\alpha}$ and $\hat{D}_{\beta}:=\hat{Q}_{\beta}-{Q}_{\alpha}$. Then the quadratic function $D_{\beta}$ is the solution of  
\[
\begin{aligned}
\underset{D}{\min}\ &\left\Vert \nabla^2 D\right\Vert_F^2 \\ 
\text{subject to} \ &D(\boldsymbol{x})=f(\boldsymbol{x})-Q_{\alpha}(\boldsymbol{x}),\ \forall \ \boldsymbol{x}  \in \mathcal{X},
\end{aligned}
\]
and the quadratic function $\hat{D}_{\beta}$ is the solution of 
\[
\begin{aligned} 
\underset{D}{\min}\ &\left\Vert \nabla^2 D\right\Vert_F^2 \\  
\text{subject to} \ &D(\boldsymbol{x})=C_{1}f(\boldsymbol{x})+C_{2}-{Q}_{\alpha}(\boldsymbol{x}),\ \forall \ \boldsymbol{x}\in \mathcal{X}.
\end{aligned}
\]

\begin{sloppypar}
We denote the parameters of the quadratic functions $D_{\beta}$ and $\hat{D}_{\beta}$ by ${\boldsymbol{\lambda}}_D\in{\mathbb{R}}^{m}$, $c_D\in{\mathbb{R}}$, $\g_D\in{\mathbb{R}}^{n}$ and ${\boldsymbol{\lambda}}_{\hat{D}}\in{\mathbb{R}}^{m}$, $c_{\hat{D}}\in{\mathbb{R}}$, $\g_{\hat{D}}\in{\mathbb{R}}^{n}$ separately. Besides, $({\boldsymbol{\lambda}}_D^{\top}, c_D, \g_D^{\top})^{\top}$ and $({\boldsymbol{\lambda}}_{\hat{D}}^{\top}, c_{\hat{D}}, \g_{\hat{D}}^{\top})^{\top}$ share the same inverse matrix of the KKT matrix, $\boldsymbol{H}$, i.e., it holds that 
\begin{align*}
\left(
{\boldsymbol{\lambda}}_D^{\top}, c_D, \g_D^{\top}\right)^{\top}=&\boldsymbol{H} \left(
f(\boldsymbol{x}_1)-Q_{\alpha}(\boldsymbol{x}_1),
\cdots,
f(\boldsymbol{x}_m)-Q_{\alpha}(\boldsymbol{x}_m),
0,
\cdots,
0
\right)^{\top},\\
\left(
     {\boldsymbol{\lambda}}_{\hat{D}}^{\top},c_{\hat{D}},g_{\hat{D}}^{\top}
\right)^{\top}=&
\boldsymbol{H}\left(
    C_{1}f(\boldsymbol{x}_1)+C_{2}-{Q}_{\alpha}(\boldsymbol{x}_1),
    \cdots,
   C_{1}f(\boldsymbol{x}_m)+C_{2}-{Q}_{\alpha}(\boldsymbol{x}_m),\right.\\
   &\left.
    0,
    \cdots,
    0
\right)^{\top},
\end{align*}
where the matrix $\boldsymbol{H}\in{\mathbb{R}}^{(m+n+1)\times(m+n+1)}$ is defined as the inverse matrix of the KKT matrix. It directly holds that
\begin{align*}
&\left(
{\boldsymbol{\lambda}}_{\hat{D}}^{\top},c_{\hat{D}},\g_{\hat{D}}^{\top}\right)^{\top}=C_{1}\left(
{\boldsymbol{\lambda}}_D^{\top},c_D,\g_D^{\top}
\right)^{\top}+\boldsymbol{H}\left(
 C_{2},
 \cdots,
 C_{2},
 0,
 \cdots,
 0
\right)^{\top}\\
&+C_{1}\boldsymbol{H}\left(
 Q_{\alpha}(\boldsymbol{x}_1),
 \cdots,
 Q_{\alpha}(\boldsymbol{x}_m),
 0,
 \cdots,
 0
\right)^{\top}
-\boldsymbol{H}\left(
 {Q}_{\alpha}(\boldsymbol{x}_1),
 \cdots,
 {Q}_{\alpha}(\boldsymbol{x}_m),
 0,
 \cdots,
 0
\right)^{\top}.
\end{align*}
Thus, it follows that
$\hat{D}_{\beta}=(C_{1} D_{\beta}+C_{2})+(C_{1}-1)\tilde{Q}$, where $\tilde{Q}$ is the solution of problem 
\[ 
\begin{aligned}
\min_{Q} \ &\left\Vert \nabla^2 Q\right\Vert_F^2\\  
\text{subject to} \  &Q(\boldsymbol{x})=Q_{\alpha}(\boldsymbol{x}),\ \forall \  \boldsymbol{x}\in \mathcal{X}.
\end{aligned}
\]
\end{sloppypar}
Then it holds that
\[
\begin{aligned}
\hat{Q}_{\beta} &={Q}_{\alpha}+\hat{D}_{\beta}={Q}_{\alpha}+(C_{1} (Q_{\beta}-Q_{\alpha})+C_{2})+(C_{1}-1)\tilde{Q}\\
&=C_1{Q}_\beta+C_2
+(C_{1}-1)(\tilde{Q}-Q_{\alpha}).
\end{aligned}
\] 
Therefore, (\ref{equationoftransformationC1f+C2}) holds and the theorem is proved.
\end{proof}

We then obtain the following corollary. 

\begin{corollary}
\label{full} 
\begin{sloppypar}
Suppose that $\mathcal{X}=\{\boldsymbol{x}_1,\cdots,\boldsymbol{x}_m\}\subset {\mathbb{R}}^{n}$ is a poised set, where \(n+1 \leqslant m<\frac{1}{2}(n+1)(n+2)\). If $L_{\alpha}$ is a linear function, then 
\begin{equation}\label{sameaffinemodelobj}
\mathcal{M}^{\mathcal{X}}_{L_{\alpha}}(C_1f+C_2)=C_1\mathcal{M}^{\mathcal{X}}_{L_{\alpha}}(f)+C_2, 
\end{equation}  
for any \(C_1,C_2\in\mathbb{R}\). 
\end{sloppypar}
\end{corollary}

\begin{proof}
Since \(\vert \mathcal{X}\vert\geqslant n+1\) and $L_{\alpha}$ is a linear function, then $\mathcal{M}^{\mathcal{X}}_{0}(L_\alpha)=L_\alpha$ according to the interpolation. Then (\ref{sameaffinemodelobj}) holds according to (\ref{equationoftransformationC1f+C2}). 
\end{proof}

Corollary refers to constructing the least Frobenius norm model, since \(\nabla^2 L_{\alpha}\) is a zero matrix. 

Notice that normally $\mathcal{M}^{\mathcal{X}}_{0}(Q_\alpha) \ne Q_{\alpha}$. Therefore the least Frobenius norm updating model of the function $C_1f+C_2$ on $\mathcal{X}$ based on $Q_{\alpha}$ may not be obtained by the same affine transformation unless $C_{1}=1$.

The above analysis also shows that the translation transformation satisfying $T\circ f=f+C_2$ for \(C_2\in \mathbb{R}\) is a model optimality-preserving transformation.

In order to analyze more about the relationship between the affine transformations and the model functions, we derive the following theorem.
\begin{theorem}
\label{anotherview}
\begin{sloppypar}
Suppose that $Q_{\alpha}$ is a quadratic function, and $\mathcal{X}=\{\boldsymbol{x}_1,\cdots,\boldsymbol{x}_m\}\subset {\mathbb{R}}^{n}$ is a poised set, where \(n+1 \leqslant m<\frac{1}{2}(n+1)(n+2)\). Given constants $\nu_1,\nu_2\in{\mathbb{R}}$, it holds that  
\begin{equation}
\label{conclusion-Q_gamma}
\mathcal{M}^{\mathcal{X}}_{\nu_1 Q_\alpha+\nu_2}(f)=\nu_1 \mathcal{M}^{\mathcal{X}}_{Q_\alpha}(f)+(1-\nu_1) \mathcal{M}^{\mathcal{X}}_{0}(f).
\end{equation}  
\end{sloppypar}
\end{theorem}
\begin{proof}
Let $Q_{\gamma}:=\mathcal{M}^{\mathcal{X}}_{\nu_1 Q_\alpha+\nu_2}(f)$, $Q_{\beta}:=\mathcal{M}^{\mathcal{X}}_{Q_\alpha}(f)$ and ${Q}_{\phi}:=\mathcal{M}^{\mathcal{X}}_{0}(f)$.  
Denote 
\begin{align*}
&Q_{\gamma}(\boldsymbol{x})-\nu_1 Q_{\alpha}(\boldsymbol{x})-\nu_2
={c}_{\gamma}+(\boldsymbol{x}-\boldsymbol{x}_0)^{\top}{\g}_{\gamma}  +\frac{1}{2} \sum_{j=1}^{m} ({{\boldsymbol{\lambda}}_{\gamma}})_{j}\left(\left(\boldsymbol{x}-\boldsymbol{x}_{0}\right)^{\top}\left(\boldsymbol{x}_{j}-\boldsymbol{x}_{0}\right)\right)^2,\\
&Q_{\beta}(\boldsymbol{x})-Q_{\alpha}(\boldsymbol{x})
={c}_{\beta}+(\boldsymbol{x}-\boldsymbol{x}_0)^{\top}{\g}_{\beta} +\frac{1}{2} \sum_{j=1}^{m} ({{\boldsymbol{\lambda}}_{\beta}})_{j}\left(\left(\boldsymbol{x}-\boldsymbol{x}_{0}\right)^{\top}\left(\boldsymbol{x}_{j}-\boldsymbol{x}_{0}\right)\right)^2,\\
&{Q}_{\phi}(\boldsymbol{x})={c}_{\phi}+(\boldsymbol{x}-\boldsymbol{x}_0)^{\top}{\g}_{\phi} +\frac{1}{2} \sum_{j=1}^{m} ({{\boldsymbol{\lambda}}_{\phi}})_{j}\left(\left(\boldsymbol{x}-\boldsymbol{x}_{0}\right)^{\top}\left(\boldsymbol{x}_{j}-\boldsymbol{x}_{0}\right)\right)^2. 
\end{align*}

We define $\q_1\in {\mathbb{R}}^{m+n+1}$ and $\q_2\in {\mathbb{R}}^{m+n+1}$ as 
\begin{align*}
\q_1&=\left(
f(\boldsymbol{x}_{1}),
\cdots,
f(\boldsymbol{x}_{m}),
0,\cdots,0
\right)^{\top},\\ 
\q_2&=\left(
Q_{\alpha}(\boldsymbol{x}_{1}),
\cdots,
Q_{\alpha}(\boldsymbol{x}_{m}),
0,
\cdots,
0\right)^{\top}.
\end{align*} 
According to the expression with the inverse matrix of the KKT matrix, it holds that 
\begin{align*}
({{\boldsymbol{\lambda}}}_{\gamma}^{\top},
{c}_{\gamma}+\nu_2,
{\g}_{\gamma}^{\top}
)^{\top}
&=\boldsymbol{H}
(\q_1-\nu_1 \q_2),\\ 
({{\boldsymbol{\lambda}}}_{\beta}^{\top},
{c}_{\beta},
{\g}_{\beta}^{\top})^{\top}&=\boldsymbol{H} (\q_1-\q_2),\\   
({{\boldsymbol{\lambda}}}_{\phi}^{\top},
{c}_{\phi},
{\g}_{\phi}^{\top})^{\top}&=\boldsymbol{H} \q_1.
\end{align*}  
\begin{sloppypar}
Thus we have
\(({{\boldsymbol{\lambda}}}_{\gamma}^{\top},{c}_{\gamma}+\nu_2,{\g}_{\gamma}^{\top})^{\top}=\nu_1({{\boldsymbol{\lambda}}}_{\beta}^{\top},{c}_{\beta},{\g}_{\beta}^{\top})^{\top}+(1-\nu_1)({{\boldsymbol{\lambda}}}_{\phi}^{\top},{c}_{\phi},{\g}_{\phi}^{\top})^{\top}.\) 
\end{sloppypar}
Therefore 
\(
Q_{\gamma}-\nu_1 Q_{\alpha}=\nu_1 (Q_{\beta}-Q_{\alpha}) +(1-\nu_1)Q_{\phi},
\)  
which implies (\ref{conclusion-Q_gamma}). The above shows that this lemma is true. 
\end{proof}

To analyze the model function corresponding to the affinely transformed objective function, we derive the following corollary based on Theorem \ref{anotherview}.
\begin{corollary}
\label{anotherview-xpc-2}
Suppose that $\mathcal{X}=\{\boldsymbol{x}_1,\cdots,\boldsymbol{x}_m\}\subset {\mathbb{R}}^{n}$ is a poised set, and $\hat{Q}_{\alpha}$ is a quadratic interpolation model of the function $f$ on $\mathcal{X}\backslash\{\boldsymbol{x}_t\}$, where \(n+1 \leqslant m<\frac{1}{2}(n+1)(n+2)\). Then ${C_{1}}\hat{Q}_{\alpha}+{C_{2}}$ is a quadratic interpolation model of $C_1f+C_2$ on $\mathcal{X}\backslash\{\boldsymbol{x}_t\}$ for any \(C_1,C_2\in\mathbb{R}\).
Moreover, it holds that 
\[
\mathcal{M}^{\mathcal{X}}_{C_1\hat{Q}_\alpha+C_2}(C_1f+C_2)-\mathcal{M}^{\mathcal{X}}_{0}(C_1f+C_2)=C_1\bigg(\mathcal{M}^{\mathcal{X}}_{\hat{Q}_{\alpha}}(C_1f+C_2)-\mathcal{M}^{\mathcal{X}}_{0}(C_1f+C_2)\bigg),
\] 
for any \(C_1,C_2\in\mathbb{R}\), where \(\mathcal{M}^{\mathcal{X}}_{0}(C_1f+C_2)\) is exactly the least Frobenius norm model of \(C_1f+C_2\). 
\end{corollary}
\begin{proof}
This is a direct consequence of Theorem \ref{anotherview}, with $\nu_1={C_1}, \nu_2={C_2}$.
\end{proof}
\begin{remark} 
Corollary \ref{anotherview-xpc-2} discusses the relationship between obtaining the least Frobenius norm updating model based on a model of the original objective function \(f\) and obtaining that based on a model of the transformed objective function \(C_1 f+C_2\). 
\end{remark}

\section{Fully linear model and convergence analysis}
\label{Convergence}

To obtain the convergence result based on the results of derivative-free trust-region algorithms, the convergence analysis in the section is for the standard provable algorithmic framework, Algorithm 10.1 in the book of Conn, Scheinberg and Vicente \cite{conn2009introduction}, but using our least Frobenius norm updating quadratc model for minimizing transformed objective functions. The only differences of the provable algorithmic framework are about the transformed output function value and the use of our model. In the other word, the function value used by the algorithm is the transformed one that strictly corresponds to the iteration step where the point being newly added. Considering that our model can provide fully linear models, we study the global convergence property for first-order critical points in detail. To catch how well the interpolation model performs given the affine transformations, we first illustrate the fully linear error constants of the least Frobenius norm updating quadratic model when the objective function is affinely transformed.

\subsection{Fully linear error constants}
\label{Fully linear error constants}

We present the following assumptions and theorem about the interpolation error between the under-determined quadratic interpolation model and the objective function with the affine transformation. 
\begin{assumption}
\label{Assumption_on_f_X_lip}
\begin{sloppypar}
We assume that $\mathcal{X}=\{\boldsymbol{x}_{1}, \cdots, \boldsymbol{x}_{m}\} \subset {\mathbb{R}}^{n}$ is a set of sample/interpolation points poised in the linear interpolation or regression sense, contained in the $\ell_2$ ball $\mathcal{B}_{\Delta}(\boldsymbol{x}_c)$ of radius $\Delta$, where \(\boldsymbol{x}_c\in\mathcal{X}\) and $n+1\leqslant \vert \mathcal{X} \vert=m < \frac{1}{2}(n+1)(n+2)$. 
\end{sloppypar}
\end{assumption}

Besides, we denote that
$\hat{\boldsymbol{{L}}}=\frac{1}{\Delta} \boldsymbol{L}=\frac{1}{\Delta}(\boldsymbol{x}_{1}-\boldsymbol{x}_{c}, \cdots, \boldsymbol{x}_{c-1}-\boldsymbol{x}_{c}, \boldsymbol{x}_{c+1}-\boldsymbol{x}_{c},\cdots, \boldsymbol{x}_{m}-\boldsymbol{x}_{c})^{\top}\in {\mathbb{R}}^{(m-1)\times n}$   and $\hat{\boldsymbol{L}}^{\dagger}=(\hat{\boldsymbol{L}}^{\top} \hat{\boldsymbol{L}})^{-1} \hat{\boldsymbol{L}}^{\top}$.  
In addition, we give the following assumption and theorem.
\begin{assumption}\label{Assumption_4_model}
We assume that \(Q_{\alpha}\) is a quadratic function, and the quadratic model function $Q_{\beta}:=\mathcal{M}^{\mathcal{X}}_{Q_{\alpha}}(f)$ is a fully linear model \cite{conn2009introduction,audet2017derivative} of the function $f$ with the constants $\kappa_g$ and $\kappa_f$, i.e., 
\begin{align*}
 \left\Vert \nabla Q_{\beta}(\boldsymbol{x})-\nabla f(\boldsymbol{x})\right\Vert_{2}  &\leqslant \kappa_{g} \Delta, \ \forall \    \boldsymbol{x} \in \mathcal{B}_{\Delta}(\boldsymbol{x}_c), \\ 
\left\vert Q_{\beta}(\boldsymbol{x})-f(\boldsymbol{x})\right\vert  &\leqslant \kappa_{f}\Delta^{2}, \ \forall \    \boldsymbol{x} \in \mathcal{B}_{\Delta}(\boldsymbol{x}_c).
\end{align*} 
 \end{assumption}
 \begin{theorem}
Suppose that Assumptions \ref{Assumption_on_f_X_lip} and \ref{Assumption_4_model} hold. Then the quadratic model function $\hat{Q}_{\beta}:=\mathcal{M}^{\mathcal{X}}_{Q_{\alpha}}(C_1f+C_2)$ is a fully linear model of $C_{1} f+C_{2}$, with constants  
\begin{align*}
    \hat{\kappa}_{g}&=\vert C_{1}\vert \kappa_g+\vert C_{1}-1\vert \left(\frac{5\sqrt{m-1}}{2}\left\Vert \hat{\boldsymbol{L}}^{\dagger}\right\Vert_{2}\left(\mu_{\alpha}+\left\Vert \nabla^2 \tilde{Q}\right\Vert_{2}\right)\right),\\
    \hat{\kappa}_{f}&=\vert C_{1}\vert  \kappa_f+\vert C_{1}-1\vert  \left(\frac{5\sqrt{m-1}}{2}\left\Vert \hat{\boldsymbol{L}}^{\dagger}\right\Vert_{2}+\frac{1}{2}\right)\left(\mu_{\alpha}+\left\Vert \nabla^2 \tilde{Q}\right\Vert_{2}\right),
\end{align*} 
for any \(C_1,C_2\in\mathbb{R}\), where $\tilde{Q}:=\mathcal{M}^{\mathcal{X}}_{0}(Q_\alpha)$, and $\mu_{\alpha}$ is the Lipschitz constant of the linear function $\nabla Q_{\alpha}$. In the other word, it holds that  
\begin{align*}
 \left\Vert \nabla \hat{Q}_{\beta}(\boldsymbol{x})-\nabla (C_1 f(\boldsymbol{x})+C_2)\right\Vert_{2}  &\leqslant \hat{\kappa}_{g} \Delta,  \ \forall \    \boldsymbol{x} \in \mathcal{B}_{\Delta}(\boldsymbol{x}_c),\\ 
\left\vert \hat{Q}_{\beta}(\boldsymbol{x})-(C_1 f(\boldsymbol{x})+C_2)\right\vert  &\leqslant \hat{\kappa}_{f}\Delta^{2}, \ \forall \    \boldsymbol{x} \in \mathcal{B}_{\Delta}(\boldsymbol{x}_c).
\end{align*} 
\end{theorem}
 
 \begin{proof}
 According to Theorem 5.4 in the book of Conn, Scheinberg and Vicente \cite{conn2009introduction}, we have 
 \begin{align*}
&\left\Vert \nabla Q_{\alpha}(\boldsymbol{x})-\nabla \tilde{Q}(\boldsymbol{x})\right\Vert_{2} \leqslant \frac{5 \sqrt{m-1}}{2}\left\Vert \hat{\boldsymbol{L}}^{\dagger}\right\Vert_{2}\left(\mu_{\alpha}+\left\Vert \nabla^2 \tilde{Q}\right\Vert_{2}\right) \Delta, \ \forall  \   \boldsymbol{x} \in \mathcal{B}_{\Delta}(\boldsymbol{x}_c),\\
&\left\vert Q_{\alpha}(\boldsymbol{x})-\tilde{Q}(\boldsymbol{x})\right\vert \leqslant \left(\frac{5 \sqrt{m-1}}{2}\left\Vert \hat{\boldsymbol{L}}^{\dagger}\right\Vert_{2}+\frac{1}{2}\right)\left(\mu_{\alpha}+\left\Vert \nabla^2 \tilde{Q}\right\Vert_{2}\right) \Delta^{2}, \ \forall  \   \boldsymbol{x} \in \mathcal{B}_{\Delta}(\boldsymbol{x}_c).
\end{align*} 
Therefore, the theorem holds combining with Theorem \ref{C1,k+C2,k} and Assumption \ref{Assumption_4_model}. 
\end{proof}

\subsection{Global convergence property for first-order critical points} 

We now turn to discuss convergence properties of our method. We assume that the fully linear error constants have uniform bounds. To avoid confusion, it should be noted here that our convergence analysis is for general positive monotonic transformation, not only for the affine transformations discussed in Section \ref{Fully linear error constants}. For the purpose of giving the convergence for first-order critical points of the provable Algorithm 10.1 in the book of Conn, Scheinberg and Vicente \cite{conn2009introduction}, but using our least Frobenius norm updating quadratc model for minimizing transformed objective functions, we assume that the transformed function $f_k$ and its gradient are  Lipschitz continuous in the corresponding domain. Details are in the following assumption.
\begin{assumption}\label{Assumption_conv_1}
\begin{sloppypar}
Suppose that the initial point \(\boldsymbol{x}_0\in {\mathbb{R}}^{n}\) and the upper bound of trust-region radius, i.e., \(\Delta_{\text{max}}\), are given. Assume that \(f\) and all \(f_k\) for \(k\in \mathbb{N}^+\) are continuously differentiable with Lipschitz continuous gradient in an open domain containing the set \(\mathcal{L}_{\text{enl}}(\boldsymbol{x}_0),\)  in which 
\[
\mathcal{L}_{\text{enl}}(\boldsymbol{x}_0)
=\underset{\boldsymbol{x} \in \mathcal{L}(\boldsymbol{x}_0)}{\cup} \mathcal{B}_{\Delta_{\text{max}}}(\boldsymbol{x}),
\] 
and 
{\(
\mathcal{L}(\boldsymbol{x}_0)=\{\boldsymbol{x} \in \mathbb{R}^{n}: f(\boldsymbol{x}) \leqslant f(\boldsymbol{x}_0)\}. 
\)}  
\end{sloppypar}
\end{assumption}

In addition, considering the minimization, we assume that each transformed function $f_k$ is bounded from below. 
\begin{assumption}\label{Assumption_conv_2} 
Suppose that \(f\) and all \(f_k\) for \(k\in \mathbb{N}^+\) are bounded from below on \(\mathcal{L}(\boldsymbol{x}_0)\), that is, there exists a constant \(\kappa_{*}\) such that, for all \(\boldsymbol{x} \in \mathcal{L}(\boldsymbol{x}_0)\), \(f(\boldsymbol{x}) \geqslant \kappa_{*}\) and \(f_k(\boldsymbol{x}) \geqslant \kappa_{*}\), \(\forall \ k \in \mathbb{N}^+\).
\end{assumption}

For simplicity, we let the Hessian of the model function, denoted as $\nabla^{2} Q_{k}$, be uniformly bounded. The following assumption shows the details.
\begin{assumption}\label{Assumption_conv_3}
There exists a constant \(\kappa_{\text{bhm}}>0\) such that, for all iterations generated by the algorithm,  
$
\Vert \nabla^2 Q_{k}\Vert_{2} \leqslant \kappa_{\text{bhm}}.
$
\end{assumption}

Following the same proof process as the convergence  analysis of Algorithm 10.1 in Chapter 10 in the book of Conn, Scheinberg and Vicente \cite{conn2009introduction}, but with the transformed function $f_{k}$ instead of the original function $f$, we can derive directly the following theorem about the convergence of the provable counterpart for minimizing the transformed objective functions. 
\begin{theorem}\label{con-new} 
Let Assumption \ref{Assumption_conv_1}, Assumption \ref{Assumption_conv_2} and Assumption \ref{Assumption_conv_3} hold. Suppose that the transformation $T_k$ is a positive monotonic transformation for each $k\in\mathbb{N}^+$, and the fully linear error constants and the gradients' Lipschitz constants of the models given by the algorithm have uniform bounds. Then
\begin{equation}
\label{convergencefk}
\lim_{k \rightarrow\infty}\ \nabla f_k(\boldsymbol{x}^{(k)})=\boldsymbol{0}
\end{equation}
holds, where $f_k(\boldsymbol{x})=T_k(f(\boldsymbol{x}))$. Moreover, it holds that 
\begin{equation}
\label{gradres}
\lim_{k \rightarrow \infty}\ \nabla f(\boldsymbol{x}^{(k)})=\boldsymbol{0}.
\end{equation}
\end{theorem}
\begin{proof}
The proof process of (\ref{convergencefk}) is the same as the convergence  analysis for trust-region methods based on derivative-free models in Section 10.4 in the book of Conn, Scheinberg and Vicente \cite{conn2009introduction}. Notice that the assumption about the uniform upper bounds of the fully linear error constants and the gradients' Lipschitz constants of the models guarantees the results corresponding to Lemma 10.5 and Lemma 10.6 in the book of Conn, Scheinberg and Vicente \cite{conn2009introduction}. Besides, given the positive monotonic transformation, we know that there exists $\varepsilon>0$, such that
\[
\underset{k\rightarrow\infty}{\lim\inf}\ \frac{\mathrm{d}f_k}{\mathrm{d}f}>\varepsilon,  
 \]
and then
\[
\nabla f_{k}(\boldsymbol{x}^{(k)})=\frac{\mathrm{d} f_{k}}{\mathrm{d} f}  \nabla f(\boldsymbol{x}^{(k)}).
\]
Therefore, (\ref{gradres}) holds. The proof is completed.
\end{proof}

{The transformations in Theorem \ref{con-new} are required to be positive monotonic, and they cover the stochastic affine transformations with positive multiplication coefficients, such as the stochastic affine transformations corresponding to (\ref{transformation}).}

Considering that Powell's software NEWUOA is a classical and efficient model-based algorithm using the least Frobenius norm updating quadratic model, Section \ref{Numerical Results} will show the results of an implementation (NEWUOA-Trans) of our method based on NEWUOA. What we need to clarify is that the convergence analysis above is based on the provable framework, but not for NEWUOA or NEWUOA-Trans, since NEWUOA is difficult to analyze because of its practical modifications and flowchart-like statements. NEWUOA-Trans is implemented for showing advantages of our model when solving problems numerically. 

NEWUOA-Trans shares the same framework with NEWUOA but updates the under-determined model by (\ref{lambda-c-g-obtain}), which is a straightforward extension of Powell's least Frobenius norm updating. In NEWUOA and NEWUOA-Trans, it is not difficult to see that the model-improvement step first attempts to replace the interpolation point that is too far from the current $\boldsymbol{x}_{\text{opt}}$ and the other interpolation points (e.g., to replace the interpolation point outside of the trust region with radius $2\Delta_{k}$ centered at $\boldsymbol{x}_{\text{opt}}$). Obviously, this can be finished in no more than $m$ steps. When all of the points in the interpolation set are close enough to each other, NEWUOA-Trans will check the poisedness of the interpolation set. The model-improvement step will find new interpolation points by maximizing the absolute values of the corresponding Lagrange polynomials or the denominator of the updating formula of the inverse matrix of the KKT matrix. The process above will not be influenced by the transformations. If the interpolation set is poised, then there is nothing about the model for NEWUOA and NEWUOA-Trans to improve. In the case where the interpolation set is not poised, one point in the set will be replaced at one step. Recalling Theorem 6.3 in the book of Conn, Scheinberg and Vicente \cite{conn2009introduction}, it holds that a poised interpolation set will be obtained in the interpolation region, and hence a fully linear model will be obtained. The model-improvement step is guaranteed to produce a fully linear model in a finite number of iterations. Therefore, we can claim that the interpolation updating of NEWUOA and NEWUOA-Trans can guarantee that the fully linear models could be constructed in a finite uniformly bounded number of steps.

\section{Numerical results}
\label{Numerical Results}

The analysis above shows that the affine transformation $T$, satisfying $T\circ f=C_1 f+C_2$, and $C_1>0$, is not a model optimality-preserving transformation, if $C_1\ne 1$. However, the affine transformation is extremely fundamental, important, and has practical applications. For example, the affine transformation exactly refers to the additive and multiplicative noise adding mechanisms with differentially private mechanisms in private black-box optimization. In fact, we have theoretically analyzed the analytic expressions and interpolation errors of the least Frobenius norm updating model for affinely transformed objective function in the previous section, and we will further observe the performance of our method through numerical experiments in this section.

\begin{sloppypar}
A derivative-free algorithm based on Powell's software NEWUOA \cite{powell2006newuoa}, for solving derivative-free optimization problems with transformed objective functions, is implemented and named NEWUOA-Trans\footnote{``-Trans'' denotes that it is designed for solving problems with transformed objective functions.}. The test codes (Fortran and MATLAB versions) can be downloaded from the online repository\footnote{\href{https://pengchengxielsec.github.io}{https://pengchengxielsec.github.io}}. The under-determined model used in NEWUOA-Trans is updated by (\ref{lambda-c-g-obtain}). This part presents the numerical results of solving derivative-free optimization problems with some kinds of transformed objective functions using NEWUOA-Trans. The numerical results demonstrate the main features and advantages of NEWUOA-Trans. NEWUOA-Trans is a robust and efficient solver for minimizing transformed objective functions and it tries to make the least change based on the successful solver NEWUOA. The modification of the codes of NEWUOA-Trans mainly occurs in the part of updating the second derivative parameters and the gradient vector of the model, although we have to admit that fixing related parts of the transformation issue is not an easy work for such interlocking codes. {The other parts in NEWUOA-Trans follow the same way in NEWUOA.}   
\end{sloppypar}

\subsection{Compared algorithms and the transformation}
Before proceeding further, we remark on the compared numerical algorithms and the transformation applied in the test problems. We apply algorithms in Table \ref{Compared algorithms} to solve the derivative-free optimization problems with transformed objective functions. The objective functions of all problems have been transformed by (\ref{transformation}). Besides, we also test NEWUOA-N, and the problems solved by NEWUOA-N have no noise, i.e., the objective functions are not transformed. The notation ``-N'' here denotes no noise, and NEWUOA-N can be regarded as the baseline in some senses, which acts as a standard answer when the noise does not exist. Furthermore, the comparison between NEWUOA-Trans and NEWUOA-N can show whether NEWUOA-Trans can reduce the transformation's influence on the objective function. Details are in Table \ref{Compared algorithms}. In NEWUOA-trans, NEWUOA-N and NEWUOA, $\rho_{\text{beg}} = 10^{-1}$, $\rho_{\text{end}} = 10^{-8}$, $m=2n+1$, and more details of the framework of NEWUOA can be seen in Fig. 1 (the outline) of Powell's work \cite{powell2006newuoa}. 

\begin{table}[htbp]  
  \centering   
    \caption{Compared algorithms\label{Compared algorithms}} 
  \begin{tabular}{lll}  
\toprule
    Algorithms& Description &  Problems
\\
\midrule
    NEWUOA-Trans& Our model &  Transformed objective\\
    NEWUOA& Powell's model \cite{powell2006newuoa} &  Transformed objective \\
     NEWUOA-N& Powell's model&  Original objective (without transformation)
\\
\botrule
    \end{tabular}
\end{table}

In our numerical experiments in Section \ref{Transformations attack NEWUOA: a simple example} and Section \ref{Performance profile}, at the $k$-th step, the output objective function value $f(\boldsymbol{x})$ at the queried points will be transformed to be 
\begin{equation}
\label{transformation}
f_k(\boldsymbol{x})=(\gamma_k+1)f(\boldsymbol{x})+C\eta_k, \ \text{for $\boldsymbol{x}$ among the $k$-th batch of queried points}, 
\end{equation}
where $\eta_k \sim \operatorname{Lap}(b_k),\ b_k>0$, and $\gamma_k \sim \operatorname{U}(-u_k, u_k)$, $0<u_k <1$. 
The probability density function of $\operatorname{Lap}(b_k)$ is 
\( 
p(x)=\frac{1}{2b_k} e^{-\frac{\vert x \vert}{b_k}}.\) 
Besides, $\operatorname{U}$ denotes the uniform distribution, and the corresponding probability density function is
\[
p(x)=\left\{
\begin{aligned}
&\frac{1}{2u_k},\ \ \text{if $x\in [-u_k, u_k]$,}\\
&\ \ \ 0,\ \ \ \text{otherwise}.
\end{aligned}
\right.
\]

\subsection{Transformations attack NEWUOA: a simple example}
\label{Transformations attack NEWUOA: a simple example}
The following simple example shows that the transformations (even the affine ones) in the objective function will lead the unmodified NEWUOA to fail. 
\begin{example}
\label{thefirstexample}
In the numerical experiments corresponding to Table \ref{table_example_if_succ}, the objective function is
\[
f(\boldsymbol{y})=\sum_{i=1}^{10} y_i^4+\sum_{i=1}^{10} y_i^2.
\] 
Notice that we denote $\boldsymbol{y}=(y_1,\cdots,y_n)^{\top}$ in the part of numerical experiments, and $n$ is 10 in this example. Besides, the initial point is $(10,\cdots,10)^{\top}$, and the constant $C=1$ in (\ref{transformation}). The analytic solution of the numerical experiments is $(0,\cdots,0)^{\top}$, of which the corresponding minimum function value is $0$. The notations $\checkmark$ and $\times$ in Table \ref{table_example_if_succ} denote whether the algorithm solves the problem successfully. The notation $\checkmark$ means that the value of $f_{\text{opt}}$, which is the numerical optimal function value obtained by each algorithm, is less than $10^{-3}$, and the notation $\times$ denotes a failure to reach that accuracy. The notation NF denotes the number of evaluated points until the iteration terminates. Besides, NEWUOA-N obtains the point with function value \(f_{\text{opt}}\) lower than \(10^{-16}\) using 990 function evaluations.

\begin{table}[htbp] 
\caption{Numerical results for Example \ref{thefirstexample}\label{table_example_if_succ}}
\begin{center}
\begin{tabular}{lllllll} 
\toprule
Transformation&\multicolumn{3}{l}{$\eta_k \sim \operatorname{Lap}(\frac{1}{k})$}
&\multicolumn{3}{l}{$\eta_k \sim \operatorname{Lap}(\frac{100}{k})$}
\\
Parameters&\multicolumn{3}{l}{$\gamma_k=0$}
&\multicolumn{3}{l}{$\gamma_k=0$}
\\
\midrule  
Algorithm  & NF & $f_{\text{opt}}$ & & NF & $f_{\text{opt}}$ &  \\ 
\midrule
  NEWUOA-Trans  
  & 1033 & 1.5626$\times 10^{-13}$ &  $\checkmark$
&1046 & 7.7485$\times 10^{-13}$ &  $\checkmark$
\\
 NEWUOA  
& 613 & 0.1375 &  $\times$
 & 348 & 7.2318 &  $\times$ 
\\
\toprule
Transformation&\multicolumn{3}{l}{$\eta_k \sim \operatorname{Lap}(\frac{10}{k})$}
&\multicolumn{3}{l}{$\eta_k=0$}
\\
Parameters&\multicolumn{3}{l}{$\gamma_k=0$}
&\multicolumn{3}{l}{$\gamma_k \sim \operatorname{U}(-\frac{1}{k},\frac{1}{k})$} 
\\
\midrule 
Algorithm  & NF & $f_{\text{opt}}$ & & NF & $f_{\text{opt}}$ & \\ 
\midrule  
  NEWUOA-Trans  
&847 & 2.6014$\times 10^{-13}$ &  $\checkmark$
  & 1055 & 3.1489$\times 10^{-13}$ &  $\checkmark$
\\
 NEWUOA 
 & 542 & 1.5818 &  $\times$ 
  &   408 & 0.7345 &  $\times$ 
\\
\toprule
Transformation&\multicolumn{3}{l}{$\eta_k \sim \operatorname{Lap}(\frac{100}{k})$}
&\multicolumn{3}{l}{$\eta_k \sim \operatorname{Lap}(\frac{100}{k})$}\\
Parameters&\multicolumn{3}{l}{$\gamma_k \sim \operatorname{U}(-\frac{1}{k},\frac{1}{k})$}
&\multicolumn{3}{l}{$\gamma_k \sim \operatorname{U}(-\frac{k}{10^4},\frac{k}{10^4})$} \\
\midrule
Algorithm  & NF & $f_{\text{opt}}$ & & NF & $f_{\text{opt}}$ &  \\ 
\midrule 
  NEWUOA-Trans   
  & 1056 & 4.1928$\times 10^{-13}$ &  $\checkmark$
  & 948 & 1.1924$\times 10^{-13}$ &  $\checkmark$ \\
 NEWUOA  
& 432 & 6.5330 &  $\times$
 &  409 & 4.0762 &  $\times$\\
\botrule
\end{tabular}
\end{center}
\end{table}

\begin{sloppypar}
From Table \ref{table_example_if_succ}, our numerical results indicate that NEWUOA can hardly solve these relatively basic and easy problems with transformed objective functions. In other words, it has unsatisfactory performances when solving derivative-free optimization problems with transformed objective functions, which is exactly caused by the influence of the transformation/noise. In addition, the results of NEWUOA-N and NEWUOA-Trans are similar, which shows that NEWUOA-Trans solves the optimization problems with transformed objective functions well, since NEWUOA-N plays a role as the baseline. 
\end{sloppypar}

\end{example}

\subsection{Performance profile and sensitivity profile}
\label{Performance profile}
\begin{sloppypar}
We use the performance profile \cite{more2009benchmarking,audet2017derivative} for the comparison of different algorithms. 
The performance profile describes the number of iterations taken by the algorithm in the algorithm set $\mathcal{A}$ to achieve a given accuracy when solving problems in a given problem set. 
We define the value 
\(
f_{\text{acc}}^{N}=\frac{f(\boldsymbol{x}^{(N)})-f(\boldsymbol{x}_{0})}{f(\boldsymbol{x}_{\text{best}})-f(\boldsymbol{x}_{0})} \in [0,1],
\)  
and the tolerance $\tau \in [0,1]$, where $\boldsymbol{x}^{(N)}$ denotes the best point found by the algorithm after obatining $N$ evaluated points, $\boldsymbol{x}_{0}$ denotes the initial point, and $\boldsymbol{x}_{\text{best}}$ denotes the best known solution. When $f_{\text{acc}}^{N} \geqslant 1-\tau$,  we say that the solution reaches the accuracy $\tau$. We give $N_{s,p}=\min\{n \in \mathbb{N},\ f_{\text{acc}}^{n}\geqslant 1-\tau \}$ and the definitions 
\[
\begin{aligned}
T_{s, p}&=\left\{\begin{aligned}
&1,  \text{ if} \ f_{\text{acc}}^{N} \geqslant 1-\tau \ \text{for some } N,\\
&0,  \text{ otherwise},
\end{aligned}\right.\\
r_{s, p}&=\left\{
\begin{aligned}
&\frac{N_{s, p}}{\min \left\{N_{\tilde{s}, p}:\ \tilde{s} \in \mathcal{A}, T_{\tilde{s}, p}=1\right\}},\ \text{if} \ T_{s, p}=1, \\
&\ \ \ \ \ \ \ \ \ \ \ \  \ \ \ \ \ +\infty,  \ \ \ \ \  \ \ \ \ \  \ \ \ \   \text{ if} \ T_{s, p}=0,
\end{aligned}\right.
\end{aligned}
\]
where $s$ is the given solver or algorithm. For the given tolerance $\tau$ and a certain problem $p$ in the problem set $\mathcal{P}$, the parameter $r_{s, p}$ shows the ratio of the number of the function evaluations using the solver $s$ divided by that using the fastest algorithm on the problem $p$. In the performance profile, $\pi_{s}(\alpha)=\frac{1}{\vert\mathcal{P}\vert}\left\vert\left\{p \in \mathcal{P}: r_{s, p} \leqslant \alpha\right\}\right\vert$, where  $\alpha \in [1, +\infty)$, and $\vert\cdot\vert$ denotes the cardinality. Notice that a higher value of $\pi_s(\alpha)$ represents solving more problems successfully. The test problems with numerical results in Fig. \ref{fig_perf_profile_comp} and Fig. \ref{fig_sens_profile_comp} are shown in Table \ref{table_test_prob}. Their dimensions are in the range 2 to 100,  and they are from classical and common unconstrained optimization test functions collections \cite{Powell2003, Conn1994, Toint1978, Li2009, Luksan2010, Andrei2008, Li1988, CUTEr,MorTesting1981}. The
upper bound of the number of function evaluations is set to be 10000. 
\end{sloppypar}
\begin{table}[htbp] 
  \centering   
    \caption{Test problems\label{table_test_prob}} 
      \setlength\tabcolsep{1pt}
  \begin{tabular}{lllllll}  
     \toprule 
    ARGLINA&  ARGLINA4&  ARGLINB &  ARGLINC & ARGTRIG\\
     ARWHEAD& BDQRTIC &  BDQRTICP &BDALUE & BROWNAL \\
      BROYDN3D &  BROYDN7D & BRYBND&  CHAINWOO & CHEBQUAD \\
      CHNROSNBZ & CHPOWELLB &CHPOWELLS & CHROSEN &COSINECUBE\\
      CURLY10 &  CURLY20 &CURLY30 &DIXMAANE & DIXMAANF \\
      DIXMAANG &DIXMAANH & DIXMAANI & DIXMAANJ &DIXMAANK \\
     DIXMAANL &DIXMAANM & DIXMAANN &DIXMAANO &DIXMAANP \\
      DQRTIC & EDENSCH  &ENGVAL1 &ERRINROS & EXPSUM \\
      EXTROSNB & EXTTET & FIROSE &FLETCBV2 & FLETCBV3 \\      
      FLETCHCR &FMINSRF2 & FREUROTH& GENBROWN & GENHUMPS \\
      GENROSE &INDEF &INTEGREQ & LIARWHD & LILIFUN3 \\
       LILIFUN4 & MOREBV &MOREBVL &NCB20 & NCB20B \\
       NONCVXU2 &NONCVXUN & NONDIA & NONDQUAR &PENALTY1 \\
       PENALTY2 &PENALTY3 &PENALTY3P &POWELLSG & POWER\\
        ROSENBROCK & SBRYBND & SBRYBNDL &SCHMVETT &SCOSINE \\
        SCOSINEL & SEROSE & SINQUAD &SPARSINE & SPARSQUR \\ 
       SPHRPTS & SPMSRTLS &SROSENBR & STMOD & TOINTGSS \\
       TOINTTRIG &TQUARTIC& TRIGSABS &TRIGSSQS & TRIROSE1 \\
     TRIROSE2 & VARDIM &WOODS & - & - \\ 
     \botrule                                  
  \end{tabular}
\end{table}

\begin{figure}[htbp]
\centering
\subfigure[$\tau=10^{-1}$\label{e}]{
\includegraphics[width=0.42\linewidth]{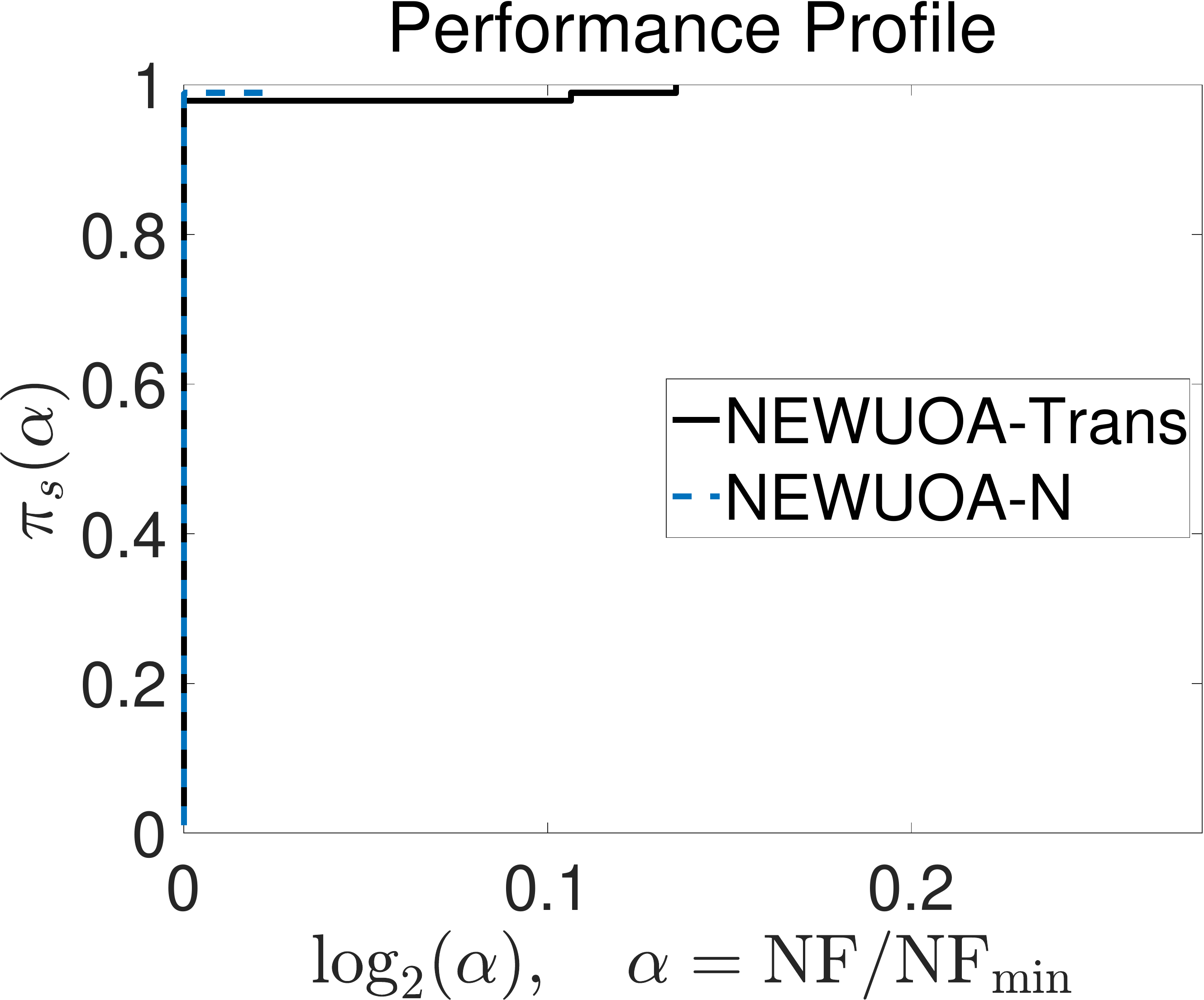}
}
\subfigure[$\tau=10^{-2}$\label{f}]{
\includegraphics[width=0.42\linewidth]{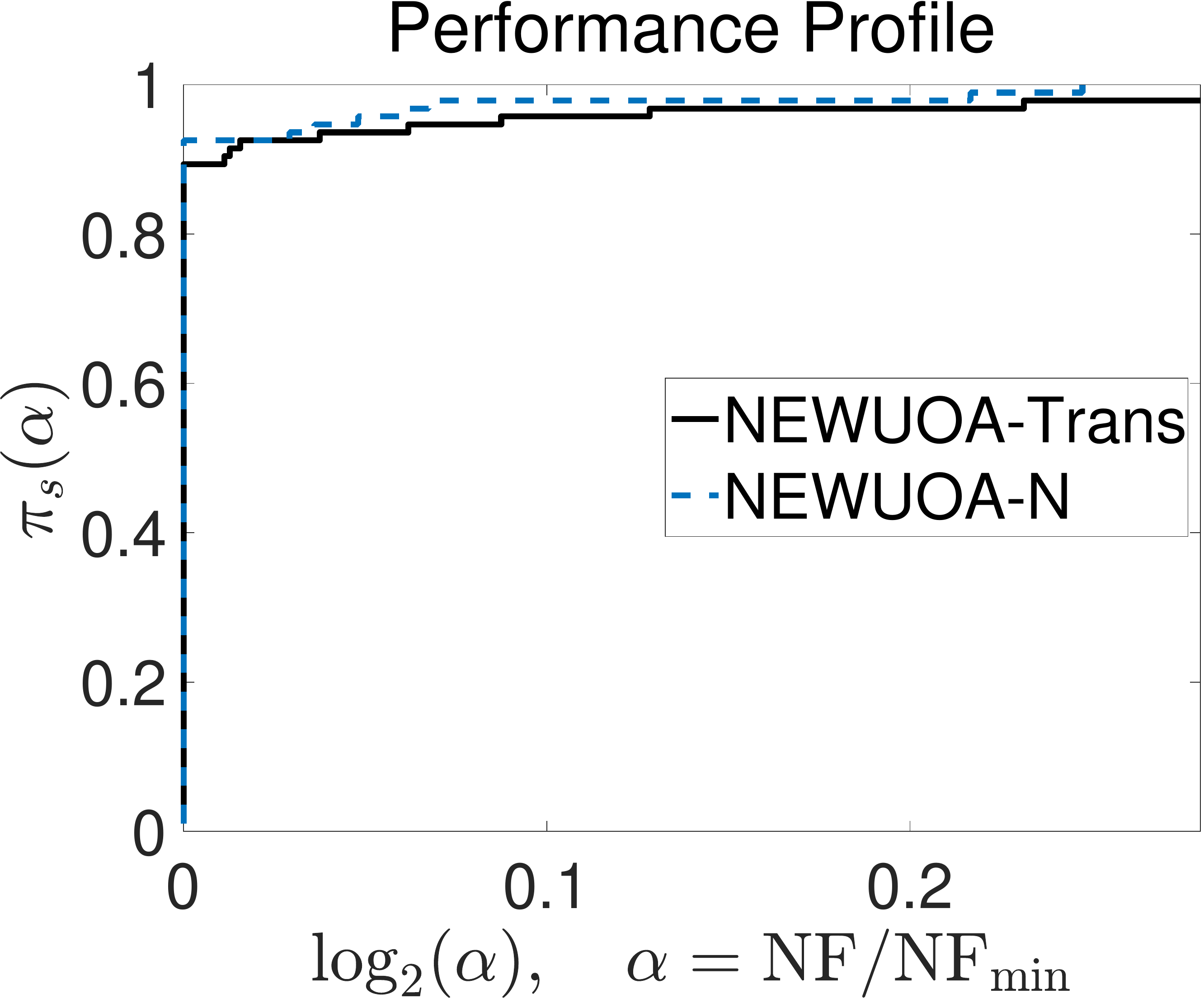}
}\\
\subfigure[$\tau=10^{-3}$\label{g}]{
\includegraphics[width=0.42\linewidth]{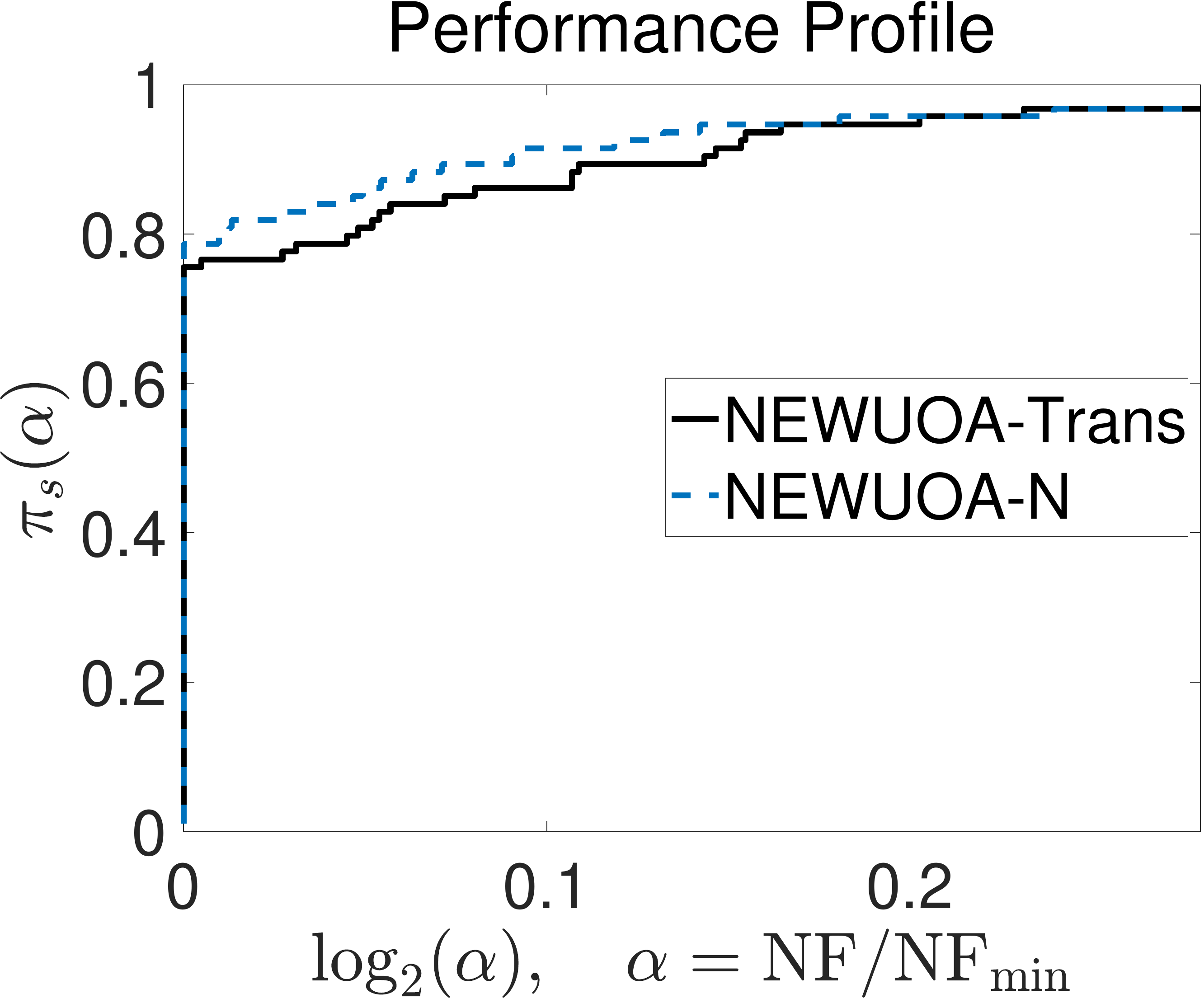}
}
\subfigure[$\tau=10^{-4}$\label{h}]{
\includegraphics[width=0.42\linewidth]{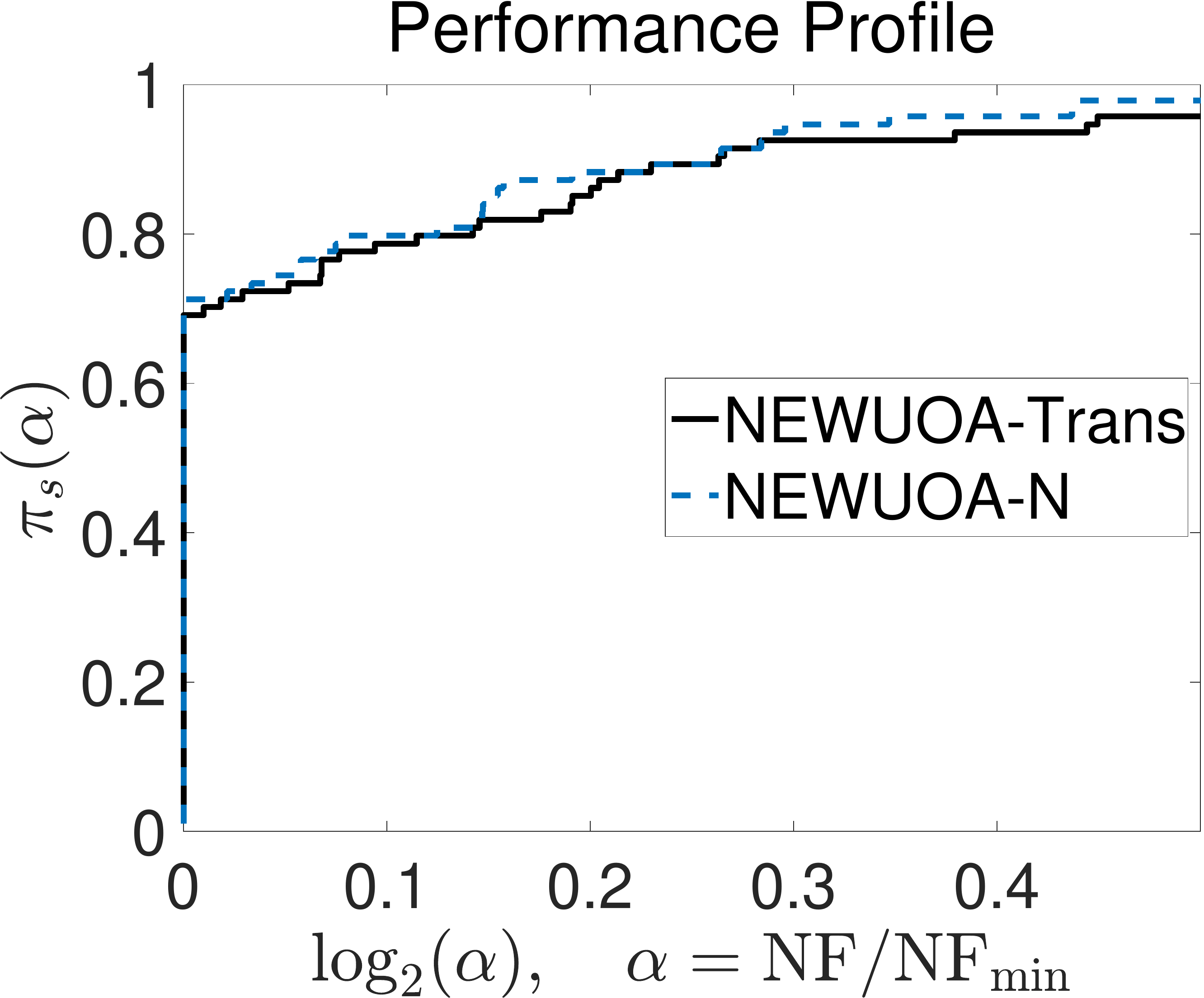}
}\\
\subfigure[$\tau=10^{-5}$\label{i}]{
\includegraphics[width=0.42\linewidth]{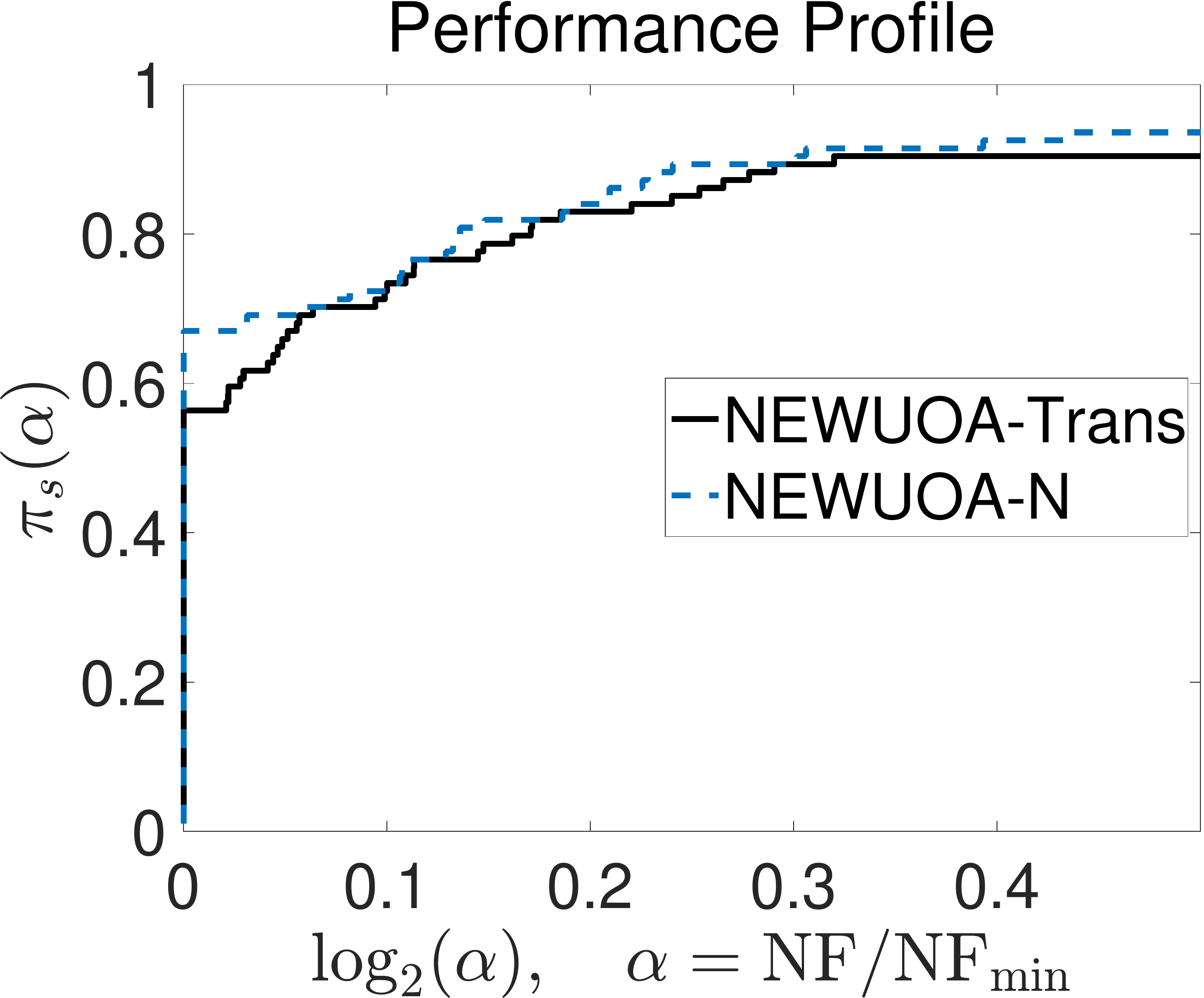}
}
\subfigure[$\tau=10^{-6}$\label{j}]{
\includegraphics[width=0.42\linewidth]{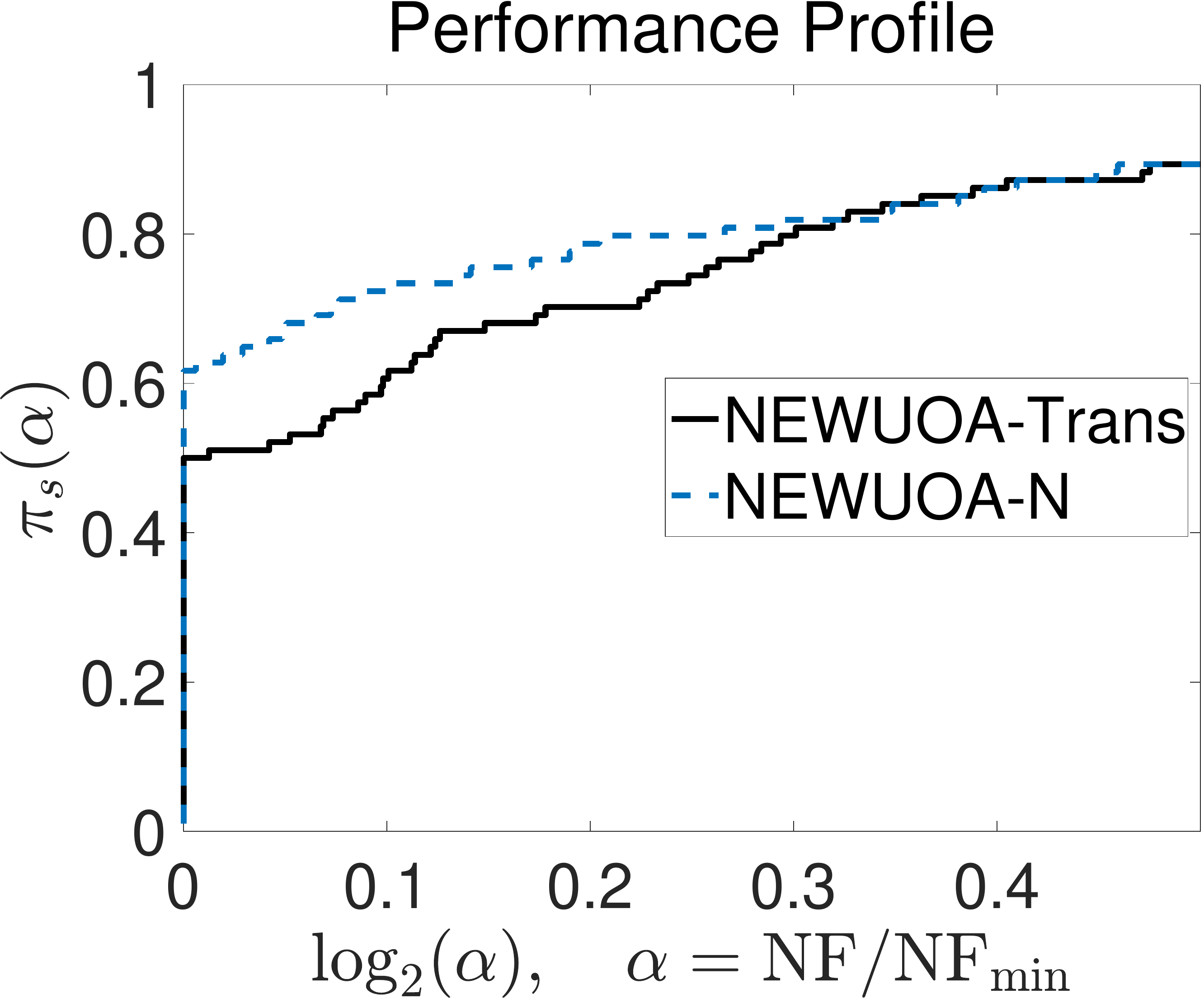}
}

\caption{The comparison of solvers solving the test problems: performance profile\label{fig_perf_profile_comp}}
\end{figure}

As shown in Example \ref{thefirstexample}, unmodified NEWUOA is not suitable for solving the DFOTO problems. We compare NEWUOA-Trans and NEWUOA-N here. 
The parameters in the transformation are $C=100$, $\eta_k\sim\operatorname{Lap}(\frac{100}{k})$, $\gamma_k\sim\operatorname{U}(-\frac{1}{k}, \frac{1}{k})$, and the two solvers share the same initial points for all problems. In Figs. \ref{e}-\ref{j}, it can be observed that when the tolerance $\tau=10^{-1},\cdots,10^{-6}$, NEWUOA-Trans and NEWUOA-N basically have the similar performance. NEWUOA-Trans  performs well for problems with transformed objective functions. The comparison in Figs. \ref{e}-\ref{j} reveals that NEWUOA-Trans can successfully solve most of the derivative-free/black-box optimization problems with transformed objective functions. The slight difference between NEWUOA-Trans and NEWUOA-N (the baseline without noise)  comes from the impact on the model from the random noise terms.

\begin{sloppypar}
In the numerical results reported in Fig. \ref{fig_sens_profile_comp}, the objective functions are chosen from test problems in Table \ref{table_test_prob}. In addition, $C=100$ and $\tau=10^{-4}$. A higher value of $\sigma_s(\alpha)$ refers to a more stable algorithm in sensitivity profile \cite{012}. Fig. \ref{fig_sens_profile_comp} depicts that the performance of NEWUOA-Trans is similar with that of NEWUOA-N, which means that the round-off error of NEWUOA-Trans is similar with that of our compared baseline NEWUOA-N. In fact, the sensitivity profile is another important criterion that evaluates the stability of the algorithms. We denote random permutation matrices as $\PP_i\in{\mathbb{R}}^{n\times n}$, $i=1,2, \cdots, M$. In the experiments, $M=100$, and an example of the random permutation matrices is $\PP_1=\left(\boldsymbol{e}_1, \boldsymbol{e}_2,\boldsymbol{e}_4,\boldsymbol{e}_3,\boldsymbol{e}_8,\boldsymbol{e}_5,\boldsymbol{e}_9,\boldsymbol{e}_{10},\boldsymbol{e}_6,\boldsymbol{e}_7\right)^{\top}$. 
Besides, we define
$
\text{NF}=\left(\text{NF}_{1}, \cdots, \text{NF}_{M}\right),
$ 
where $\text{NF}_{i}$ denotes the number of function evaluations when solving the corresponding problem
\(
\min_{\boldsymbol{x} \in {\mathbb{R}}^{n}} f(\PP_i\boldsymbol{x}).
\)  
We define 
$
\operatorname{mean}(\text{NF})$ as $\frac{1}{M} \sum_{i}^{M} \text{NF}_{i}$, and the standard deviation $\operatorname{std}(\text{NF})$ as $\sqrt{\frac{1}{M} \sum_{i}^{M}(\text{NF}_{i}-\operatorname{mean}(\text{NF}))^{2}}$. In the performance profile, we apply $\operatorname{std}(\text{NF})$, corresponding to solving the problem $p$ with the solver $s$, instead of $N_{s,p}$, to obtain $\sigma_s(\alpha)$ instead of $\pi_s(\alpha)$, and then we obtain the sensitivity profile. 
\end{sloppypar}

\begin{figure}[htbp]
\centering
{
\includegraphics[width=0.5\linewidth]{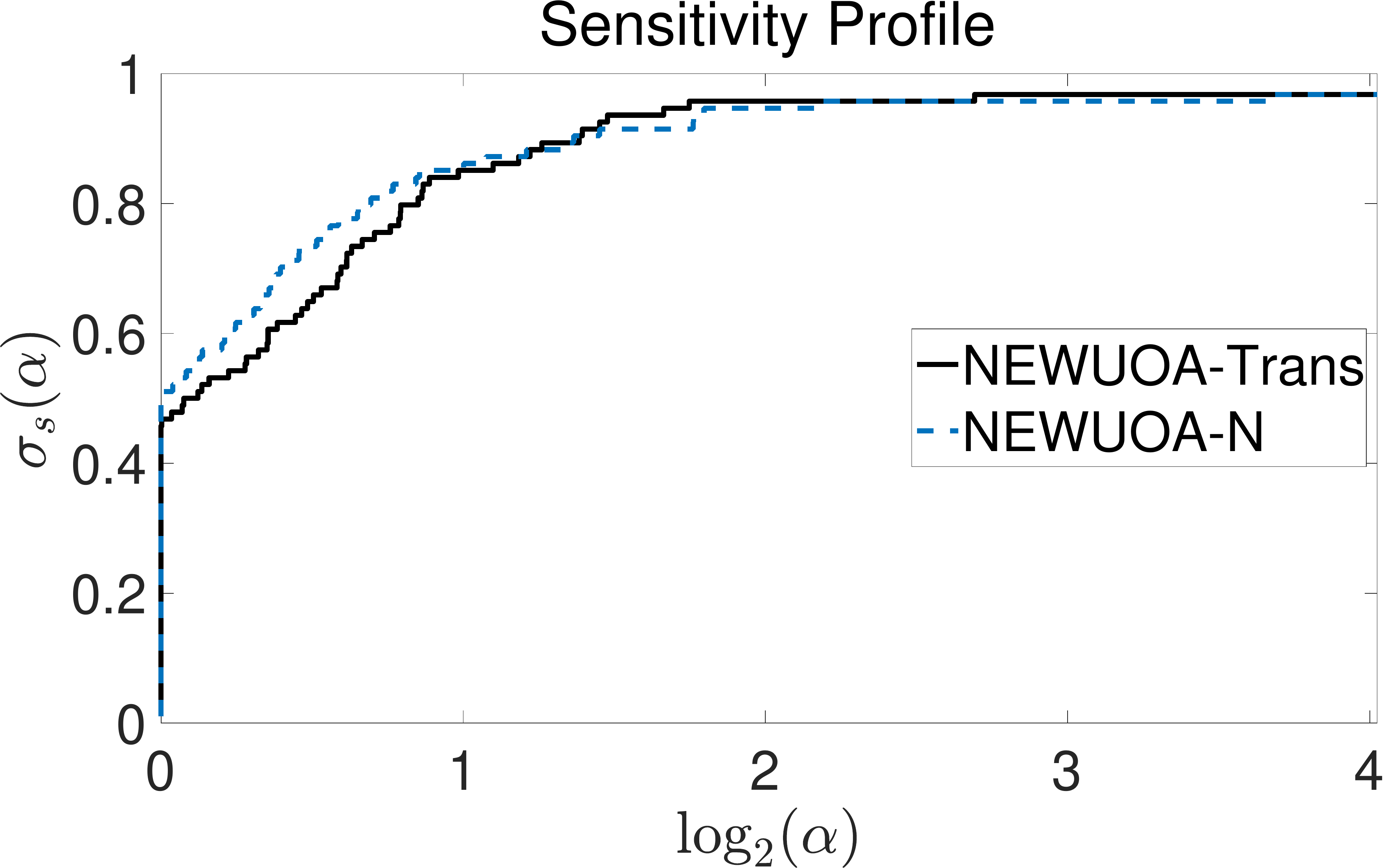}
}

\caption{The comparison of solvers solving the test problems: sensitivity profile\label{fig_sens_profile_comp}}
\end{figure}

We can see that the interference of the transformation barely affects NEWUOA-Trans since the curves of NEWUOA-Trans and NEWUOA-N are close to each other, while such interference indeed affects the other algorithms. The performance of the algorithm is consistent with the theoretical analysis. 

\begin{sloppypar}
Our numerical experiments show that NEWUOA-Trans can efficiently and stably solve most derivative-free optimization problems with transformed objective functions. The performances of NEWUOA-Trans and NEWUOA-N are similar, which shows that NEWUOA-Trans solves the optimization problems with transformed objective functions well, eliminating the interference from the transformation.  
\end{sloppypar}

\subsection{Real-world experiments}\label{TWT}

We test our method for private engineering optimal design of the space traveling wave tube (TWT) \cite{xbg2}. The TWT is a critical kind of the vacuum electronic device \cite{xbg1}. It affects the signal quality and intensity and is mainly used in communication, transportation, navigation, meteorological measurement, forecast, and other fields. The space TWT is the most expensive one with the highest requirements in the TWT family. 

Because of the special working environment of the space TWT, it is necessary to save the energy of the satellite as much as possible. Therefore, the efficiency of the space TWT is very critical. In the design of the space TWT, the design parameters are important to the indicator, efficiency, which is exactly what we want to maximize, while its expression is unknown, and the relationship between the parameters and the objective function is hard to analyze. Moreover, the data of the space TWT is difficult to obtain. Most of the data about its efficiency performance needs the experiment or the simulation, which refers to a very high cost, and some of them are encrypted in practice. It usually costs very much to obtain an objective function value with the given parameters even though engineers use the approximate values of the efficiency obtained by the simulation software, which still usually costs a long time. Therefore, it is suitable to apply derivative-free methods to solve such an expensive optimization problem. As an important industrial product related to the security, copyrights, interests, and so on, the true value of the efficiency of some special kind of space traveling wave tube is supposed to be encrypted during the optimization process (especially to the public and the optimization solver, which is the third party beyond the user/buyer). Thus optimizing the efficiency of the special kind of TWT is a private black-box optimization problem, and can belong to DFOTO.

We do the following numerical experiments, and the test data of the space TWT is provided by Beijing vacuum electronics research institute (BVERI). The purpose of the numerical experiment here is to find the input parameters for privately designing space TWT that has the highest efficiency, which is formulated to the private unconstrained DFO problem\footnote{Such unconstrained problem is formulated by BVERI, and some weak constraints are deleted in advance for simplicity, which does not influent the optimization.}
\[
\max\limits_{\mathcal{P}_{\text{input}}}\ {\rm Efficiency}\left(\mathcal{P}_{\text{input}}\right),
\]
where $\mathcal{P}_{\text{input}}$ is a $10$-dimensional vector, which denotes the parameters for designing the space TWT. In order to protect the true values of the efficiency, the designers of the space TWT apply the stochastic affine transformation at each queried step to encrypt the true function value in the query process. {The basic query process follows Assumption \ref{Qassumption} and Table \ref{table1}.} We apply NEWUOA-Trans to solve this problem, and we choose the initial input \(10^{-1}\times(2,\cdots,2)^{\top}\) as the initial point, and $\rho_{\text{beg}}=10^{-1}$, $\rho_{\text{end}}=10^{-4}$. NEWUOA-Trans terminates after 226 iterations. The iteration process can be seen in Table \ref{table_iter_dist_real}, which presents the $\ell_2$ norm distance between the best iteration point at the $k$-th iteration, i.e., \(\boldsymbol{x}_{\text{best}}^{(k)}\), and the final solution $\boldsymbol{x}^*$.

\begin{table}[htbp]
  \centering   
  \caption{The $\ell_2$ norm distance between the best iteration point at the \(k\)-th step and the final solution: \(\|\boldsymbol{x}_{\text{\rm best}}^{(k)}-\boldsymbol{x}^*\|_2\)\label{table_iter_dist_real}}  
  \begin{tabular}{lllllllll}     
     \toprule
Iter.  & 10 & 20 & 30 & 40 & 50 & 60 & 70 & 80 \\
Dist.  & 
84.4577&
42.6845&
19.0530&
13.7870&
7.7990&
4.7825&
0.9851&
0.8116 \\
\midrule 
Iter.  & 90& 100& 110  &120 & 130 & 140 & 150 & 160\\
Dist. &
0.7110&
0.5525&
0.5106&
0.4705&
0.4034&
0.3035&
0.1318&
0.1102\\
\midrule 
Iter.   & 170 & 180 & 190& 200& 210& 220 \\
Dist. &
0.0800&
0.0560&
0.0370&
0.0102&
0.0025&
0\\ 
 \botrule                        
  \end{tabular}
\end{table}

To check our results, we can simulate the design with the optimized designing parameters with the large electromagnetic computing software CST, and find that, at the working frequency band, the efficiency is increased from the best setting based on the expert knowledge and experience, as shown in Table \ref{efficiency}.
\begin{table}[htbp]
\caption{Efficiency increment\label{efficiency}}
\centering
\begin{tabular}{llll}
 \toprule
Frequency point (GHz)& $94$ & $97$ & $100$ 
\\
\midrule
Efficiency increment (\textperthousand) & $53 $ & $62$ & $66$ \\
 \botrule
\end{tabular}
\end{table}

According to the evaluation mechanism of BVERI, parameters obtained by solving the corresponding DFOTO problem using NEWUOA-Trans achieve the maximum efficiency for the design of this special space TWT with acceptable number of function evaluations, and the corresponding maximum efficiency is satisfactory in the industry. This shows that our method has strong practicality and high accuracy. The industry also potentially favors the privacy protection property of our method, since the data providers can only (and only need to) output the transformed function values. The above preliminary application also encourages us to apply our method in broader fields.

\section{Conclusions}
\begin{sloppypar}

Before closing this article, we propose the following continuously-transformed problem beyond the query/evaluation oracle in Assumption \ref{Qassumption}, which is a new challenging mathematical programming problem. 
\begin{OP}[derivative-free methods minimizing ``moving-target'' type objective function]\label{moving-target}
Try to design practical numerical optimization algorithms for the unconstrained problem (\ref{transformproblem}), where $f(\boldsymbol{x},t)$ is the practical output value of the black-box function $f$ at $\boldsymbol{x}\in{\mathbb{R}}^n$, given the $t\in\mathbb{R}$, where $t$ depends exactly on the query/evaluation order of the current point $\boldsymbol{x}$, or $t$ can be regarded as the time. In the other word, the set of queried function values will hold the form as 
\(\{f(\boldsymbol{x}_1,t_1),f(\boldsymbol{x}_2,t_2),\cdots,f(\boldsymbol{x}_k,t_k),\cdots\}\), where \(t_k\) denotes the discrete query time. 
\end{OP}

In conclusion, this article discusses derivative-free optimization with transformed objective function (DFOTO). We present a method with a corresponding query scheme. The existence of model optimality-preserving transformations in addition to the translation transformation is proved for the strictly convex model with its unique minimizer in the trust region. The necessary and sufficient condition for the transformed function value, referring to model optimality-preserving transformations, is proposed. We obtain the corresponding quadratic model of the affinely transformed objective function and prove that some positive monotonical transformations, even affine transformations with positive multiplication coefficients, are not model optimality-preserving transformations. We also provide the interpolation error analysis for the corresponding model function, given affinely transformed objective functions. Convergence analysis for first-order critical points is given. Numerical results of test problems and the real-world problem support our method and analysis.

This article is a preliminary attempt to throw light on this direction, and much about derivative-free optimization with transformation remains to be studied. More applications with details of building the least Frobenius norm updating quadratic models for derivative-free optimization with transformed objective functions (e.g., private black-box optimization with noise-adding mechanisms) will be investigated as future work. The fully linear error constants discussed in Section \ref{Convergence} are allowed to change at each iteration rather than being uniform bounds, so they could potentially grow unboundedly across iterations if the multiplication coefficient $C_1$ is unbounded. Thus the convergence analysis of solving the problems with transformations under a weaker assumption than that being used in Section \ref{Convergence} is still an open and challenging problem. Open Problem \ref{moving-target} about minimizing ``moving-target'' type objective function is also interesting and valuable. 

\end{sloppypar}




\bibliography{JORSC-bibliography}

\end{document}